\documentclass[12pt,a4paper,final]{amsart}
\usepackage{version}
\usepackage[notcite,notref]{showkeys}
\usepackage{amsmath,amsthm,amssymb,latexsym,epsfig,graphicx,subfigure}
\numberwithin{equation}{section}

\newtheorem{thm}{Theorem}[section]
\newtheorem{lm}[thm]{Lemma}

\newtheorem{cor}[thm]{Corollary}
 
\newtheorem{pr}[thm]{Proposition}
\theoremstyle{definition}
\newtheorem{df}[thm]{Definition}
\theoremstyle{definition}

\newtheorem{rems}[thm]{Remarks}
\newtheorem{ex}[thm]{Example}

\newcommand{\Rn}{\mathbb{R}^{n}}
\newcommand{\Pn}{\mathbb{P}^{n}}
\newcommand{\p}{\mathbb{P}}
\newcommand{\hn}{\mathbb{R}^{n-1}}
\newcommand{\R}{\mathbb{R}}

\renewcommand{\hn}{\mathbb{H}^n}

\renewcommand{\H}{\mathcal{H}}

\newcommand{\h}{\mathbb{H}}

\newcommand {\grtrsim} {\ {\raise-.5ex\hbox{$\buildrel>\over\sim$}}\ }
\newcommand{\e}{\varepsilon}

\newcommand{\khii}{\text{\lower -.4ex\hbox{$\chi$}}}
\DeclareMathOperator{\spt}{spt}

\DeclareMathOperator{\tanm}{Tan}

\renewcommand{\a}{\alpha}

\newcommand{\restrict}{\begin{picture}(12,12)
                       \put(2,0){\line(1,0){8}}
                       \put(2,0){\line(0,1){8}}
                      \end{picture}}

\begin{document} 
\title {Parabolic rectifiability, tangent planes and tangent measures}
\author{Pertti Mattila}


\begin{abstract} 
We define rectifiability in $\R^{n}\times\R$ with a parabolic metric in terms of $C^1$ graphs and Lipschitz graphs with small Lipschitz constants and we characterize it in terms of approximate tangent planes and tangent measures. We also discuss  relations between the parabolic rectifiability and other notions of rectifiability.
\end{abstract}

\maketitle
{\small\tableofcontents}
\section{Introduction}
Let $\|\cdot\|$ be the parabolic 'norm'
$$\|(x,t)\|=\sqrt{|x|^2+|t|}$$
in $\Pn=\Rn\times\R$ and $d$ the corresponding metric $d(p,q)=\|p-q\|$. Here, and later, $|x|$ is the Euclidean norm of $x$. The following is the main result of the paper:

\begin{thm}\label{main}
Let $m$ and $n$ be positive integers with $0<m<n+2$. Let $E\subset \Pn$ be $\H^m$ measurable and $\mathcal H^m(E)<\infty$. Then the following are equivalent:
\begin{itemize}
\item[(1)] $E$ is $C^1G$ $m$-rectifiable.
\item[(2)] $E$ is LG $m$-rectifiable.
\item[(3)] $E$ has an approximate tangent $m$-plane at $\H^m$ almost all of its points.
\item[(4)] For $\H^m$ almost all $a\in E$\ there is an $m$-flat measure $\lambda_a$ such that\\ $\tanm(\H^m\restrict E,a)=\{c\lambda_a: 0<c<\infty\}$.
\item[(5)] For $\H^m$ almost all $a\in E$\ $\H^m\restrict E$ has a unique tangent measure at $a$.
\end{itemize}
\end{thm}

\begin{df}
We shall say that $E$ is \emph{parabolic $m$-rectifiable} if one, and hence all, of the conditions (1) to (5) holds.
\end{df}

This requires some explanations, the precise definitions will be given later. First, $\H^m$ is the $m$-dimensional Hausdorff measure related to the metric $d$. Notice that the Hausdorff dimension of $\Pn$ is $n+2$. $C^1G$ refers to parabolic  $C^1$ graphs and LG to Lipschitz graphs and $E$ is defined to be $C^1G$, respectively LG, $m$-rectifiable if $\H^m$ almost all of it can be covered with countably many $C^1$ graphs, respectively Lipschitz graphs with arbitrarily small Lipschitz constants, over subsets of homogeneous linear subspaces of Hausdorff dimension $m$. In Example \ref{ex} we shall show that we need small constants; Lipschitz graphs themselves are not always parabolic rectifiable. Homogeneous means that the linear subspaces are invariant under the dilations 
$$\delta_r:\Pn\to\Pn, \delta_r(x,t)=(rx,r^2t), r>0.$$ Then they are precisely the linear $m$-dimensional subspaces of $\Rn\times\{0\}$ and linear $(m-1)$-dimensional subspaces of $\R^{n+1}$ containing $\{0\}\times\R$. A homogeneous linear subspace $V$ is an approximate tangent $m$-plane of $E$ at $a$ if near $a$\ $E$ lies close to $V+a$ in a metric measure-theoretic sense. $m$-flat measures are just constant multiples of the Lebesgue measures on homogeneous linear subspaces. The elements of $\tanm(\H^m\restrict E,a)$ are the tangent measures at $a$ of the restriction of $\H^m$ to $E$,\ $\H^m\restrict E$, that is, the weak limits of the normalized blow-ups of $\H^m\restrict E$. Finally, the uniqueness of the tangent measures means uniqueness up to multiplication by positive constants. 

Rectifiable sets were defined in Euclidean spaces in the plane by Besicovitch in the 1920s and in general dimensions by Federer in 1947 in \cite{Fed47}.  According to Federer's definition $E\subset\Rn$ is $m$-rectifiable if almost all of it can be covered with countably many Lipschitz images of subsets of $\R^m$. It follows from Rademacher's theorem that in Euclidean spaces this is equivalent to (1) and (2). Example \ref{ex1} shows that Euclidean $m$-rectifiable sets need not be parabolic  $m$-rectifiable, but parabolic  $m$-rectifiable sets are Euclidean $m$-rectifiable by Theorem \ref{eucpr}. Moreover, Euclidean $m$-rectifiable sets are parabolic  $(m+1)$-rectifiable by Theorem \ref{eucpr1}. In Euclidean spaces the equivalence of (1), (2) and (3) was proved by Federer and of (4) to these by Preiss, who introduced the tangent measures in \cite{Pre87}. The equivalence of (4) and (5) was proved in \cite{Mat05} in metric groups in a much more general setting. Ambrosio and Kirchheim \cite{AK00} used Federer's definition in general metric spaces and they proved an analogue of $(2)\Longleftrightarrow (3)$. 

Rectifiability in Heisenberg groups, and in more general Carnot groups, has been studied by many people starting with the pioneering work of Franchi, Serapioni and Serra Cassano \cite{FSS01}. An analogue of Theorem \ref{main} in Heisenberg groups was proved in \cite{MSS10} and in general Carnot groups  by Antonelli and Merlo in \cite{AM21}. Our parabolic space is a special case of a homogeneous group, which is not a Carnot group. Idu, Magnani and Maiale \cite{IMM20} proved in general homogeneous groups a result similar to Theorem \ref{main}, but only for horizontal sets. 

Although rectifiability can be  defined in general metric spaces by Federer's definition, it is evident that it is not a natural definition in spaces with a special very non-Euclidean structure, such as Heisenberg groups and $\Pn$. In particular, we shall see in Theorem \ref{Lip0} that in $\Pn$ all Federer  $(n+1)$-rectifiable sets have Hausdorff $n+1$ measure zero. It is not always clear what a natural definition should be. One criterion could be the condition (5) in Theorem \ref{main}: rectifiable sets should look the same at all small scales around typical points, whenever such a statement makes sense. So if we wish to have an optimally large collection with this property, we have a right definition. 

As far as I know, except for \cite{IMM20} mentioned above, non-quantitative rectifiability in the parabolic setting has not been considered before in the literature. But uniform rectifiability in codimension 1 was introduced by Hoffman, Lewis and Nystr\"om in \cite{HLN03} and  \cite{HLN04}. Recently it has been investigated  by many people, see, for example, \cite{BHHLN20} and \cite{BHHLN21}. Relations between Euclidean uniform rectifiability and boundary behaviour of solutions of elliptic equations and related singular integrals have been investigated by great success during the recent decades. A motivation for parabolic uniform rectifiability comes, I believe, from the desire to extend results of this type to parabolic equations. 

After some preliminaries in Section \ref{Preli}  in Section \ref{Lipgraph} we characterize parabolic Lipschitz graphs with parabolic cones and relate them to rectifiability. In Section \ref{APPtan} we introduce approximate tangent planes and tangent measures and finish the proof of the equivalence of (2), (3), (4) and (5) in Theorem \ref{main}. The arguments follow the Euclidean pattern with several changes. In Section \ref{Paradiff} we introduce parabolic $C^1$ graphs and complete the proof of Theorem \ref{main}. There we also prove a Rademacher-type theorem for parabolic rectifiable Lipschitz graphs, Theorem \ref{diffthm}. In Section \ref{euclid} we compare Euclidean and parabolic rectifiability. We show that parabolic $m$-rectifiability implies Euclidean $m$-rectifiability and Euclidean $m$-rectifiability implies parabolic $(m+1)$-rectifiability. The converse statements are false by examples in Section \ref{example}. In Section \ref{Other} we discuss uniform, Ambrosio-Kirchheim and Heisenberg rectifiability and their relations to parabolic rectifiability. In Theorem \ref{Lip0} we show that all codimension 1 Ambrosio-Kirchheim rectifiable sets have parabolic Hausdorff measure zero. In Section \ref{example} we construct several examples. Example \ref{ex} shows that the vertical  Lipschitz graphs  need not be parabolic rectifiable, in particular, Lipschitz does not imply almost everywhere differentiable in this setting.. In addition,  this gives a simple example of a vertical Lipschitz graph which is not uniformly rectifiable, an earlier such an example is due to Lewis and Silver in \cite{LS88}. A modification 
of this, Example \ref{ex4}, shows that the qualitative notion of rectifiability related to the above mentioned uniform rectifiability differs from our parabolic rectifiability. 

Example \ref{ex} means a failure of a strong type of Rademacher's theorem. Similar failure in some Carnot groups for intrinsic Lipschitz graphs was verified in \cite{JNV21}, see Section \ref{Heisenberg}.

\emph{Acknowledgement.} I would like to thank referee for many useful comments.

\section{Preliminaries}\label{Preli}

The metric $d$ and the dilations $\delta_r$ in $\Pn$ are as defined in Introduction. Then $d(\delta_rp,\delta_rq)=rd(p,q)$. In $\Pn$,\ $d(A)$ stands for the diameter of $A$,\ $d(A,B)$ for the minimal distance between the sets $A$ and $B$ and $d(p,A)$ for the distance from a point $p$ to a set $A$. The closed ball with centre $p\in \Pn$ and radius $r>0$ is denoted by $B(p,r)$ and the open ball by $U(p,r)$. The corresponding concepts with the Euclidean distance are denoted by $d_E$ and $B_E$. 

We denote by $\mathcal L^n$ the Lebesgue measure in the Euclidean $n$-space $\Rn, n\geq 1.$ For $s>0$ the $s$-dimensional Hausdorff measure $\H^s$ in $\Pn$  is defined by

$$\H^s(A)=\lim_{\delta\to 0}\inf\{\sum_{i=1}^{\infty}d(E_i)^s: A\subset \bigcup_{i=1}^{\infty}E_i, d(E_i)<\delta\}.$$

Let $\H^s_E$ be the Euclidean Hausdorff measure defined as $\H^s$ but the parabolic diameter replaced by the Euclidean one.

\begin{lm}\label{Hs} There is positive constant $c(n,s)$ such that $c(n,s)\H^{s+1}(A)\leq \H^s_E(A)\leq \H^s(A)$ for $A\subset\Pn$.
\end{lm}

\begin{proof}
The right hand inequality is trivial, because the Euclidean diameter is at most the parabolic for small sets. The left hand inequality follows from the observation that any Euclidean ball of radius $r$ can be covered roughly with $1/r$ parabolic balls of radius $r$.
\end{proof}

A set $A\subset\Pn$ is called AD-$m$-\emph{regular} if there are $0<c<C<\infty$ such that $cr^m\leq \H^m(A\cap B(p,r))\leq Cr^m$ for $p\in A$ and $0<r<d(A)$.

The \emph{upper and lower $s$-densities} of $A\subset \Pn$ at $a\in\Pn$ are defined by

$$\Theta^{\ast s}(A,p)=\limsup_{r\to 0}(2r)^{-s}\mathcal H^s(A\cap B(p,r)),$$ 
$$\Theta_{\ast}^s(A,p)=\liminf_{r\to 0}(2r)^{-s}\mathcal H^s(A\cap B(p,r)).$$ 
If these agree, we define the \emph{$s$-density} $\Theta^{s}(A,p)$ as their common value.

\begin{thm}\label{dens} If $A\subset\Pn$ is $\H^s$ measurable and $\mathcal H^s(A)<\infty$, then
\begin{itemize}
\item[(1)] $2^{-s}\leq\Theta^{\ast s}(A,p)\leq 1$ for $\mathcal H^s$ almost all $p\in A,$
\item[(2)] $\Theta^{\ast s}(A,p)=0$ for $\mathcal H^s$ almost all $p\in \Pn\setminus A.$\end{itemize}
\end{thm} 

For the proof, see \cite{Fed69}, 2.10.19(2),(4),(5), also the proofs of \cite{EG92}, Section 2.3, work in general metric spaces.

For $0<m\leq n$ let $H(n,m)$ be the set of linear $m$-dimensional subspaces of $\Rn\times\{0\}\subset\Pn$, the horizontal subspaces. For $1<m\leq n+1$ let $V(n,m)$ be the set of linear $(m-1)$-dimensional subspaces of $\Pn$ containing $\{0\}\times\R$, the vertical subspaces. Then $V(n,m)$ consists of the orthogonal complements in $\R^{n+1}$ of the horizontal subspaces in $H(n,n+2-m)$. We also set $P(n,m)=H(n,m)\cup V(n,m)$. Then $P(n,1)=H(n,1)$ and $P(n,n+1)=V(n,n+1)$. For $V\in P(n,m), V^{\perp}$ is the orthogonal complement of $V$ in $\R^{n+1}$.

Denote by $H$ the horizontal plane $\{t=0\}$ and by $T$ the $t$-axis $\{x=0\}$. Then any vertical plane $V$ can be written as $V=V\cap H + T$. For $V\in P(n,m)$ let $P_V$ be the orthogonal projection onto $V$, that is, $P_V(x,t)=(P_{V}(x),0)$, when $V\in H(n,m)$, and $P_V(x,t)=(P_{V\cap H}(x),t)$, when $V\in V(n,m)$, where $P_{V}$ and $P_{V\cap H}$ also denote the standard orthogonal projections in $\Rn$. Then 
\begin{equation}\label{norm}
\|p\|^2=\|P_V(p)\|^2+\|P_{V^\perp}(p)\|^2.
\end{equation}

All these projections $P_V:\Pn\to\Pn$ are 1-Lipschitz mappings. We equip $P(n,m)$ with the compact metric $d$,\ $d(V,W)=\|P_{V\cap H}-P_{W\cap H}\|$, where $\|\cdot\|$ is the operator norm: for any linear map $\Lambda:\Pn\to\Pn, \|\Lambda\|$ is the smallest number $L$ such that $\|\Lambda(p)\|\leq L\|p\|$ for $p\in \Pn$.

The results of this paper remain valid with small technical changes for other commonly used metrics, too, for example the ones induced by $\|(x,t)\|=(|x|^4+|t|^2)^{1/4}, \|(x,t)\|=|x|+\sqrt{|t|}$ and $\|(x,t)\|=\max\{|x|,\sqrt{|t|}\}$, but the formula \eqref{norm} makes $d$ more convenient. In particular, the validity of the statements of Theorem \ref{main} is independent of the metrics as long as they are bilipschitz equivalent.

We say that a subgroup of $(\Pn,+)$ is \emph{homogeneous} if it is closed and invariant under the dilations $\delta_r, r>0$. It is easy to check that the homogeneous subgroups are exactly the elements of $P(n,m)$ together with $\Pn$ and $\{0\}$. 

\begin{lm}\label{LebH}
For $V\in P(n,m)$ there is a positive number $p(V)$ such that $\H^m\restrict V=\frac{p(V)}{\mathcal L_V(V\cap B(0,1))}\mathcal L_V$, where $\mathcal L_V$ is the Lebesgue measure on $V$. In particular, $\H^m(V\cap B(p,r))=p(V)r^m$ for $p\in V, r>0$. Moreover,  $p(V)=2^m$  for $V\in H(n,m)$ and for $2\leq m\leq n+1$ there is a positive number $v(m), 1\leq v(m)\leq 2^m,$ such that $p(V)=v(m)$  for $V\in V(n,m)$.
\end{lm}

\begin{proof}
The Hausdorff measure $\H^m$ on $V\in P(n,m)$ agrees with a constant multiple of the Lebesgue measure of $V$ by the uniqueness of Haar measures, which gives $p(V)$. If $V\in H(n,m)$, it is isometric with $\R^m$, whence $p(V)=2^m$, see, e.g., \cite{EG92}, Theorem 2 in Section 2.2. If $V\in V(n,m)$ and $m\geq 3$ (for $n=2$ we have only the $t$-axis), it is isometric with $\p^{m-2}$ which gives $v(m)$. By Theorem \ref{dens}(1) with $A=V$, $1\leq v(m)\leq 2^m$. The claim for balls in the horizontal case is clear. For $V\in V(n,m)$,\ $\mathcal L_V(B(0,r))=\mathcal L_V(V\cap B(0,1))r^{n+2}$ because $B(0,r)=\delta_r(B(0,1))$ and $\det(\delta_r)=r^{n+2}$.
\end{proof}

I don't know the precise value of $v(m)$ but I expect it to be strictly less than $2^{m}$.

As $\H^m$ on $V\in P(n,m)$ is doubling, Vitali's covering theorem for it holds, which together with Lemma \ref{LebH}  yields the density theorem:

\begin{lm}\label{Lebdens}
If  $V\in P(n,m)$ and $A\subset V$, then  $\Theta^{m}(A,p)=2^{-m}p(V)$ for $\H^m$ almost all $p\in A$.
\end{lm}

 We call the measures $c\H^m\restrict V, V\in P(n,m), c>0$, \emph{$m$-flat}, and horizontal $m$-flat or vertical $m$-flat depending on whether $V\in H(n,m)$ or $V\in V(n,m)$.

We shall denote by $f_{\#}\mu$ the push-forward of a measure $\mu$  under a map  $f: f_{\#}\mu(A)= \mu(f^{-1}(A))$. The restriction of $\mu$ to a set $A$, $\mu\restrict A$, is defined by $\mu\restrict A(B)=\mu(A\cap B)$. The notation $\ll$ stands for absolute continuity. The support of a measure $\mu$ is denoted by $\spt\mu$. The closure of a set $A$ is $\overline{A}$.

By the notation $a\lesssim b$ we mean that $a\leq Cb$ for some constant $C$. The dependence of $C$, if not denoted explicitly, should be clear from the context.

\section{Lipschitz graphs, cones and rectifiable sets}\label{Lipgraph}
We now define Lipschitz graphs over homogeneous planes and characterize them with cones.

\begin{df}\label{lipgraph}We say that $G\subset \Pn$ is an \emph{$m$-Lipschitz graph}  if there exist $0<L<\infty, A\subset V\in P(n,m)$ and $g:A\to V^{\perp}$ such that $\|g(x)-g(y)\|\leq L\|x-y\|$ for $x,y\in A$ and $G=\{x+g(x):x\in A\}$. Then we also call $G$ an $m$-Lipschitz graph or  $(m,L)$-Lipschitz graph over $V$. We write $G=G_g$.

We say that $G\subset \Pn$ is a \emph{horizontal} (resp. \emph{vertical}) $m$-Lipschitz graph, or a horizontal (resp. vertical) $(m,L)$-Lipschitz graph,  if the above holds with $V\in H(n,m)$ (resp. $V\in V(n,m))$.
\end{df}

\begin{rems}\label{graphrem}

(1) If above $g$ is $L$-Lipschitz, then the map $x\mapsto x+g(x)$  is $\sqrt{L^2+1}$-bilipschitz. Hence $G_g$ is AD-$m$-regular, if $A$ is. 

(2) Let $P$ be the projection onto the $t$-axis; $P(x,t)=t$. If $G$ is a horizontal Lipschitz graph, then $\mathcal L^1(P(G))=0$. This follows applying the following Euclidean fact to the last coordinate function of $g$ with $G=G_g$: 

If $f:A\to\R, A\subset\R^m$, satisfies $|f(x)-f(y)|\leq C|x-y|^2$ for $x,y\in A$, then $\mathcal L^1(f(A))=0$. This is a consequence of a Sard theorem. A simple proof for a special case which suffices here can be found in \cite{Mat95}, Theorem 7.6.

(3) Immediately by the definitions, the horizontal Lipschitz graphs are locally Euclidean Lipschitz graphs and the Euclidean  Lipschitz graphs over vertical planes are locally vertical Lipschitz graphs.
\end{rems}

\begin{lm}\label{HVlm1}
If $G$ is a vertical $m$-Lipschitz graph, then the Euclidean Hausdorff dimension of $G$ is at most $m-1/2$. In particular, $\H^m_E(G)=0$. 
\end{lm}

\begin{proof}
This is a well-known fact about H\"older graphs, but I give the simple argument. Let $g:A\to V^{\perp}, A\subset V\in V(n,m)$, be $L$-Lipschitz and $d(A)<1$. Cover $A$ with  Euclidean cubes $Q_i\subset V, i=1,\dots,N_r\lesssim r^{1-m}$ of side-length $r$. Then $\{x+g(x):x\in Q_i\}$ can be covered with Euclidean balls $B_{i,j}, j=1,\dots,M_r\lesssim r^{-1/2}$, of radius $r$. Hence
$$\sum_{i,j}d_E(B_{i,j})^{m-1/2}\lesssim r^{1-m}r^{-1/2}r^{m-1/2}=1,$$
from which the lemma follows
\end{proof}

\begin{cor}\label{HVlm}
If $G$ is a horizontal $m$-Lipschitz graph and $G'$ is a vertical $m$-Lipschitz graph, then $\H^m(G\cap G')=0$.
\end{cor}

\begin{proof}
By Lemma \ref{HVlm1} $\H_E^m(G')=0$. Let $P(x,t)=(x,0)$ for $(x,t)\in\Pn$. Then $\H^m(P(G\cap G'))=\H_E^m(P(G\cap G'))=0$. So $G\cap G'$ is a horizontal $m$-Lipschitz graph over a set of measure zero and therefore $\H^m(G\cap G')=0$.
\end{proof}
 
For $V\in P(n,m), p\in\Pn$ and $0<s<1$ define the open cone around $V$,
\begin{equation}\label{cone}
X(p,V,s)=\{q\in\Pn:\|P_{V^{\perp}}(q-p)\|<s\|q-p\|\}=\{q\in\Pn:d(q-p,V)<s\|q-p\|\},\end{equation}
and for $r>0$ set $X(p,r,V,s)=X(p,V,s)\cap B(p,r)$.

Then by \eqref{norm}
\begin{equation}\label{M2eq4}
X(p,V^{\perp},\sqrt{1-s^2}) = \Pn\setminus \overline{X(p,V,s)}.
\end{equation}

Notice that  for small $s$  the parabolic horizontal cones are much narrower and the vertical much wider than the Euclidean cones.

\begin{lm}\label{conegraph}
Let $V\in P(n,m)$. 
\begin{itemize}
\item[(1)] If $L>0$ and $G$ is  an $(m,L)$-Lipschitz graph over $V$, then\\ $(G\setminus\{p\})\setminus X(p,V,s)=\emptyset$ for $p\in G$ and for $s>L$. 
\item[(2)] If $0<s<1$ and $(G\setminus\{p\})\setminus X(p,V,s)=\emptyset$ for $p\in G$, then $G$ is  an $(m,L)$-Lipschitz graph over $V$ with $L=\tfrac{s}{\sqrt{1-s^2}}$. 
\end{itemize}
\end{lm}
\begin{proof} If $s>L$ and $G_g$ is  an $(m,L)$-Lipschitz graph over $V$, then for $p=x+g(x), q=y+g(y)\in G_g, p\neq q,$ we have 
$$\|P_{V^{\perp}}(q-p)\|=\|g(x)-g(y)\|\leq L\|q-p\|<s\|q-p\|,$$
so $q\in X(p,V,s)$, and (1) follows.

If $G$ satisfies the assumption of (2), then $\|P_{V^{\perp}}(q-p)\|<s\|q-p\|$ for $p,q\in G, p\neq q$, so by \eqref{norm} $\|P_V(q-p)\|>\sqrt{1-s^2}\|q-p\|$. Hence $P_V|G$ is injective and its inverse is $(\sqrt{1-s^2})^{-1}$-Lipschitz. Letting $g=P_{V^{\perp}}\circ (P_V|G)^{-1}$ we have $G=G_g$ with $g$\ $s (\sqrt{1-s^2})^{-1}$-Lipschitz. 

\end{proof}

Notice that by \eqref{M2eq4} (2) can be written as: If $0<t<1$ and $(G\setminus\{p\})\cap\overline{ X(p,V^{\perp},t)}=\emptyset$ for $p\in G$, then $G$ is  an $(m,L)$-Lipschitz graph over $V$ with $L=\tfrac{\sqrt{1-t^2}}{t}$. 

We now define

\begin{df}\label{m-rect}
A set $E\subset\Pn$ is \emph{LG $m$-rectifiable} if for every $0<L<\infty$ there are $(m,L)$-Lipschitz graphs $G_i, i=1,2,\dots,$ such that
$$\H^m(E\setminus\bigcup_{i=1}^{\infty}G_i)=0.$$
A set $E\subset\Pn$ is \emph{purely LG $m$-unrectifiable} if $\H^m(E\cap F)=0$ for every LG $m$-rectifiable set $F\subset \R^n$.
\end{df}

\begin{df}\label{m,L-rect}
For $0<L<\infty$ a set $E\subset\Pn$ is \emph{$(m,L)$-rectifiable} if  there are $(m,L)$-Lipschitz graphs $G_i, i=1,2,\dots,$ such that
$$\H^m(E\setminus\bigcup_{i=1}^{\infty}G_i)=0.$$
$E\subset\Pn$ is \emph{purely $(m,L)$-unrectifiable} if $\H^m(E\cap F)=0$ for every $(m,L)$-rectifiable set $F\subset \Pn$.\
\end{df}

Thus $E$ is LG $m$-rectifiable if and only if it is $(m,L)$-rectifiable for all $L>0$ and it is purely $(m,L)$-unrectifiable if and only if $\H^m(E\cap G)=0$ for every $(m,L)$-Lipschitz graph $G$. 

In the above definitions it is enough to consider Lipschitz graphs $G_g, g:A\to V^{\perp}$, where $A$ is closed, since for arbitrary $A\subset V$  $g$ can be extended as a Lipschitz map with the same constant to the closure of $A$. If $V$ is vertical we could take $A=V$, because real valued (with the standard metric) Lipschitz functions on any metric space can be extended without increasing the constant, see \cite{Fed69}, 2.10.44. But if $V$ is horizontal, the last coordinate function of every Lipschitz map $g:V\to V^{\perp}$ must be constant, so we cannot always extend.

As in the Euclidean case, see Theorem 15.6 in \cite{Mat95}, any  $E\subset \Pn$ with $\mathcal H^m(E)<\infty$ has an $\H^m$ almost unique decomposition $E=R\cup P$ where $R$ is LG $m$-rectifiable and $P$ is purely LG $m$-unrectifiable. Similarly for $(m,L)$-rectifiability.

\begin{thm}\label{mLthm}
Let $E\subset \Pn$ be $\H^m$ measurable and $\mathcal H^m(E)<\infty$. 
\begin{itemize}
\item[(1)] Let $s>L>0$. If $E$ is $(m,L)$-rectifiable, then for $\H^m$ almost all $a\in E$ there is $V\in P(n,m)$ such that
\begin{equation}\label{M2eq}
\lim_{r\to 0}r^{-m}\H^m(E\cap B(a,r)\setminus X(a,V,s))=0.
\end{equation}
\item[(2)] Let $0<s<1$ and $L>\tfrac{s}{\sqrt{1-s^2}}$. If for $\H^m$ almost all $a\in E$ there is $V\in P(n,m)$ such that \eqref{M2eq} holds, then $E$ is $(m,L)$-rectifiable.
\end{itemize}
\end{thm}

\begin{proof}
To prove (1) suppose that $E$ is $(m,L)$-rectifiable and $s>L$. Then there are $(m,L)$-Lipschitz graphs $G_i, i=1,2,\dots,$ over $V_i\in P(n,m)$ such that $\H^m(E\setminus\bigcup_{i=1}^{\infty}G_i)=0.$ By Theorem \ref{dens} for $\H^m$ almost all $a\in G_i, \Theta^{\ast m}(E\setminus G_i,a)=0$. By Lemma \ref{conegraph},\  $(G_i\setminus\{a\})\setminus X(a,V_i,s)=\emptyset$ for $a\in G_i$. Thus 
\begin{equation}\label{M2eq1}
\lim_{r\to 0}r^{-m}\H^m(E\cap B(a,r)\setminus X(a,V_i,s))=0.\end{equation}

For (2) we use the following lemma. Its proof below is a slight modification of the proof of Lemma 15.14 in \cite{Mat95}.

\begin{lm}\label{M2lemma}
Let $V\in P(n,m)$ and let $t, L,\delta$ and $\lambda$ be positive numbers with $0<t<1$ and $L>\tfrac{\sqrt{1-t^2}}{t}$.
If $E\subset\Pn$ is purely $(m,L)$-unrectifiable and 
\begin{equation}\label{M2eq2}
\H^m(E\cap X(a,r,V^{\perp},t))\leq \lambda r^m\ \text{for}\ a\in E, 0<r<\delta,
\end{equation}
then
\begin{equation}\label{M2eq3}
\H^m(E\cap B(a,\delta/6))\leq C(m,L,t)\lambda \delta^m\ \text{for}\ a\in \Pn.\end{equation}
\end{lm}
\begin{proof} We may assume that $E\subset B(a,\delta/6)$.
Choose $0<\kappa<\sigma<\tau <1, 0<\kappa<(\tau-\sigma)/2$, depending on $L$ and $t$ such that $L>\tfrac{\sqrt{1-(\sigma t)^2}}{\sigma t}$. Let 
$F$ be the set of $p\in E$ for which  
$$E\cap X(p,V^{\perp},\sigma t) \neq \emptyset.$$
As $E$ is purely $(m,L)$-unrectifiable we have by Lemma \ref{conegraph}(2) and the remark after it that
 $\H^m(E\setminus F)=0$. Let
$$h(p)=\sup\{\|q-p\|: q\in E\cap X(p,V^{\perp},\sigma t)\}\in (0,\delta/3]\ \text{for}\ p\in F$$
and choose $p^{\ast}\in E\cap X(p,V^{\perp},\sigma t)$ with $\tau h(p)< \|p-p^{\ast}\|\leq h(p)$. Letting 
$$C_p=(P_V)^{-1}(P_V(B(p,\kappa th(p))))=(P_V)^{-1}(V\cap B(P_V(p),\kappa th(p)))$$ we have
\begin{equation}\label{M2eq5}
F\cap C_p\subset X(p,2h(p),V^{\perp},t)\cup X(p^{\ast},2h(p),V^{\perp},t)\ \text{for}\ p\in F.
\end{equation}
To prove this let $q\in F\cap C_p$. Then $\|P_V(q-p)\|\leq\kappa th(p)$. We have $\|q-p\|\leq h(p)$, because if $\|q-p\|> h(p)$, then $\|P_V(q-p)\|\leq\kappa th(p)<\sigma t\|q-p\|$, so $q\in E\cap X(p,V^{\perp},\sigma t)$, whence $\|q-p\|\leq h(p)$. 

Suppose $q\not\in X(p^{\ast},2h(p),V^{\perp},t)$. Then
\begin{align*}
&t\|q-p^{\ast}\|\leq \|P_V(q-p^{\ast})\|\leq \|P_V(p-p^{\ast})\|+ \|P_V(q-p)\|\\
&\leq \sigma t\|p-p^{\ast}\| + \kappa th(p) \leq (\sigma+\kappa)th(p),
\end{align*}
whence $\|q-p^{\ast}\|\leq  (\sigma+\kappa)h(p)$. As $\|p-p^{\ast}\|> \tau h(p)$, we obtain
$$\|p-q\|\geq \|p-p^{\ast}\| - \|p^{\ast}-q\|> (\tau-\sigma-\kappa)h(p)\geq \|P_V(p-q)\|/t,$$
because $\|P_V(p-q)\|\leq \kappa th(p)\leq (\tau-\sigma-\kappa)th(p)$, since $2\kappa<\tau-\sigma.$ Consequently, $q\in X(p,2h(p),V^{\perp},t)$ and we have verified \eqref{M2eq5}.

By \eqref{M2eq2} we now have 
$$\H^m(F\cap C_p)\leq 2\lambda (2h(p))^m.$$
By a standard covering lemma, see, e.g., \cite{Mat95}, Theorem 2.1, we can cover $P_V(F)$ with countably many balls $P_VB(p,\kappa th(p)), p\in S\subset F,$ such that the balls $P_V(B(p,\kappa th(p)/5))$ are disjoint. Then $F\subset \cup_{p\in S}C_p$. As $\H^m(V\cap B(q,r))\leq 2^mr^m$ for $q\in V, r>0$,  by Lemma \ref{LebH}, we have
\begin{align*}
&\H^m(E)=\H^m(F) \leq \sum_{p\in S}\H^m(F\cap C_p) \leq \sum_{p\in S}10^m\kappa^{-m}2\lambda (\kappa h(p)/5)^m\\
&\leq C(m,\kappa,t)\lambda\H^m(V\cap B(P_V(a),\delta))\leq C(m,L,t)\lambda\delta^m.
\end{align*}
\end{proof}

Let us return to the proof of Theorem \ref{mLthm}. To prove (2) suppose that for $\H^m$ almost all $a\in E$ there is $V\in P(n,m)$ such that \eqref{M2eq} holds. Let $0<s<u$ be such that $L>\tfrac{u}{\sqrt{1-u^2}}$. Suppose that $E$ is purely $(m,L)$-unrectifiable. It is enough to show that then $\H^m(E)=0$. There is $\eta>0$ such that $X(a,V,s)\subset X(a,V',u)$ for all  $V,V'\in P(n,m)$ with $d(V,V')<\eta$ and for all $a\in\Pn$. Splitting $E$ into a finite union, we may assume that there is $V\in P(n,m)$ such that \eqref{M2eq} holds with $s$ replaced by $u$ for this $V$ and for all $a\in E$. Let $\lambda>0$. We can assume that there is a positive number $\delta$ such that  with $t=\sqrt{1-u^2}$,
\begin{equation}\label{M2eq6}
\H^m(E\cap  X(a,r,V^{\perp},t))\leq\H^m(E\cap B(a,r)\setminus X(a,V,u))\leq \lambda r^m\end{equation}
for $a\in E, 0<r<\delta$, because $E$ is a countable union of such sets. Then by Lemma \ref{M2lemma}, 
$$\H^m(E\cap B(a,r))\leq C(m,L,t)\lambda r^m\ \text{for}\ a\in \Pn, 0<r<\delta/6.$$
Choosing $\lambda$ sufficiently small, we have $\Theta^{\ast m}(E,a)<2^{-m}$ for all $a\in E$. Thus $\H^m(E)=0$ by Theorem \ref{dens}.
\end{proof}

\section{Approximate tangent planes and tangent measures}\label{APPtan}

\begin{df}\label{m-apprtan}
A plane $V\in P(n,m)$ is an \emph{approximate tangent $m$-plane}  of a set $E\subset\Pn$ at a point $a\in\Pn$ if for every $s>0$,
$$\lim_{r\to 0}r^{-m}\H^m(E\cap B(a,r)\setminus X(a,V,s))=0.$$
\end{df} 

Recall that here $m$ is the Hausdorff dimension of $V$ and the linear dimension of $V$ is either $m$ or $m-1$.

If $E\subset\Pn$ is AD-$m$-regular, then $V\in P(n,m)$ is an approximate tangent $m$-plane of $E$ at $a$ if and only for every $0<s<1$ there is $r>0$ such that $E\cap B(a,r)\setminus X(a,V,s)=\emptyset$. This is pretty obvious.

\begin{lm}\label{apprtandist}
Let $E\subset \Pn$ be $\H^m$ measurable and $\mathcal H^m(E)<\infty$. If $V\in P(n,m)$ is an approximate tangent $m$-plane of $E$ at $a$, then for every $\delta>0$,
\begin{equation}\label{apptaneq}
\lim_{r\to 0}r^{-m}\H^m(E\cap B(a,r)\cap \{p:d(p-a,V)\geq\delta r\})=0.
\end{equation}
Conversely, for  $\H^m$ almost all $a\in E$, if $V\in P(n,m)$ and \eqref{apptaneq} holds for every $\delta>0$, then $V$ is an approximate tangent $m$-plane of $E$ at $a$.
\end{lm}

\begin{proof}
The first statement is obvious, since $X(a,r,V,s)\subset B(a,r)\cap \{p:d(p-a,V)\leq sr\}$. For the converse, observe that for all $\e>0$, 
$$B(a,r)\setminus X(a,V,s) \subset (B(a,r)\cap \{p:d(p-a,V)\geq\e sr\})\cup B(a,\e r),$$ 
and use the fact that for $\H^m$ almost all $a\in E, \Theta^{\ast m}(E,a)\leq 1$ by Theorem \ref{dens}.
\end{proof}
An approximate tangent $m$-plane need not be unique at every point, but it is at almost all points:

\begin{lm}\label{apprunique}
Let $E\subset \Pn$ be $\H^m$ measurable and $\mathcal H^m(E)<\infty$. Then at $\H^m$ almost all  points of $E$ where an approximate tangent $m$-plane of $E$ exists, it is unique.
\end{lm}

\begin{proof}
Let us first check that for any $V,V'\in P(n,m), V\neq V',$ there is $\eta(s)>0$ for $s>0$ with $\lim_{s\to 0}\eta(s)=0$ such that for $a,p\in\Pn$,
$$d(p-a,V)<s\|p-a\|\ \text{and}\ d(p-a,V')<s\|p-a\| \Rightarrow d(p-a,V\cap V')<\eta(s)\|p-a\|.$$
To prove this we may assume that $a=0$ and, using the dilations $\delta_{1/\|p\|}$, that $\|p\|=1$. If the claim is false, there are 
$\eta>0, p_i\in\Pn, \|p_i\|=1,$ and $s_i>0$ with $\lim_{i\to\infty}s_i=0$ such that $d(p_i,V)<s_i, d(p_i,V')<s_i$ and  $d(p_i,V\cap V')\geq\eta.$ Taking a converging subsequence of $(p_i)$ we obtain a contradiction.

By Theorem \ref{dens}(1) there are $\H^m$ measurable sets $E_j \subset E, j=1,2,\dots,$ and positive numbers $r_j$ such that $\H^m(E\setminus \cup_jE_j)=0$ and
\begin{equation}\label{appr1}
\H^m(E_j\cap B(p,r)) \leq (3r)^m\ \text{for}\ p\in E_j, 0<r<r_j.
\end{equation}
By Theorem \ref{dens}(2) it is enough to prove that the approximate tangent $m$-planes of each $E_j$ are unique almost everywhere in $E_j$. 

Suppose that $E_j$ has approximate tangent $m$-planes $V,V'\in P(n,m)$ at $a\in E_j$ with $V\neq V'$. Let $0<r<r_j, 0<s<1$ and let $\eta(s)$ be as above. Then for $r>0$, 
$$X(a,r,V,s)\cap X(a,r,V',s)\subset \{p\in\Pn:d(p-a,V\cap V')\leq \eta(s)r\}.$$
Let $k<m$ be the Hausdorff dimension of $V\cap V'$. By Lemma \ref{LebH} for any $q\in V\cap V', \rho>0,$\ $\H^k(V\cap V'\cap B(q,\rho))=p(V\cap V')\rho^k$, which implies that any ball of radius $2r$ in $V\cap V'$ can be covered with $N\lesssim \eta (s)^{-k}$ balls of radius $\eta(s)r$. Thus $E_j\cap B(a,r)\cap \{p\in\Pn:d(p-a,V\cap V')\leq \eta(s)r\}$ can be covered with $N$ balls of radius $4\eta(s)r$ centered in $E_j$. Hence by \eqref{appr1}
\begin{align*}
&\H^m(E_j\cap B(a,r)\cap \{p\in\Pn:d(p-a,V\cap V')\leq \eta(s)r\})\\
&\leq N(12\eta(s)r)^{m}\lesssim \eta(s)^{m-k}r^m<r^m,\end{align*}
if $s$ is sufficiently small. Now $E_j\cap B(a,r)$ is contained in the union of the sets $E_j\cap B(a,r)\setminus X(a,V,s), E_j\cap B(a,r)\setminus X(a,V',s)$  and $E_j\cap X(a,r,V,s)\cap X(a,r,V',s)$. From these we deduce that $\Theta^{\ast m}(E_j,a) < 2^{-m}$ and the lemma follows from Theorem \ref{dens}.
\end{proof}

Tangent measures were introduced by Preiss in \cite{Pre87} to solve the density characterization of rectifiability. 

Define

$$T_{a,r}(p)=\delta_{1/r}(p-a),\ p,a\in\Pn, r>0.$$

So $T_{a,r}$ blows up the ball $B(a,r)$ to the unit ball. Now we also blow up measures.

\begin{df}
Let $\mu$ be a Radon measure on $\Pn$. A non-zero Radon measure $\nu$ is called a \emph{tangent measure} of $\mu$ at $a\in\Pn$ if there are sequences $(c_i)$ and $(r_i)$ of positive numbers such that $r_i\to 0$ and $c_iT_{a,r_i\#}\mu\to\nu$ weakly. We denote the set of tangent measures of $\mu$ at $a$ by $\tanm(\mu,a)$.
\end{df}

Tangent measures tell us how the measure looks locally. We say that $\mu$ has a unique tangent measure at $a$ if there is $\nu$ such that $\tanm(\mu,a)=\{c\nu: 0<c<\infty\}$.

The following result was proved in \cite{Mat05}, Theorem 3.2, in metric groups in a much more general setting:

\begin{thm}\label{M2}
Let $\mu$ be a Radon measure on $\Pn$. Then the following are equivalent:
\begin{itemize}
\item[(1)] For $\mu$ almost all $a\in\Pn$\ there is an $m$-flat measure $\nu$ such that $\tanm(\mu,a)=\{c\nu: 0<c<\infty\}$.
\item[(2)] For $\mu$ almost all $a\in\Pn$\ $\mu$ has a unique tangent measure at $a$.
\end{itemize}
\end{thm}

We shall use the following lemma. 
\begin{lm}\label{tanlemma}
Let $\mu$ and $\nu$ be Radon measures on $\Pn$.
\begin{itemize}
\item[(1)] If $B\subset\Pn$ is $\mu$ measurable, $a\in\spt\mu$ and $\lim_{r\to 0}\frac{\mu(B(a,r)\setminus B)}{\mu(B(a,r)}=0$, then $\tanm(\mu\restrict B,a)=\tanm(\mu,a)$.
\item[(2)] If $\nu \ll \mu$, then $\tanm(\nu,a)=\tanm(\mu,a)$ for $\nu$ almost all $a\in\Pn$.
\end{itemize}
\end{lm}

In Euclidean spaces this is the same as Lemmas 14.5 and 14.6 in \cite{Mat95}. The same proof works for (1), and it works for any other homogeneous metric in place of $d$, too. The proof of Lemma 14.6 of \cite{Mat95} works for (2) provided the density theorem for $\mu$ and $\nu$ holds. This would require Besicovitch's covering theorem and I don't know if it is true with the metric $d$ in $\Pn$. But it is true for some other homogeneous metric by a very general result of Le Donne and Rigot \cite{LR19} and also by the special case $\|(x,t)\|=\max\{|x|,\sqrt{|t|}\}$ treated by Itoh \cite{Ito18}. Statement (2) then follows, since it only involves dilations and not the metric explicitly.

I believe that Besicovitch's covering theorem in $\Pn$ is valid also with $d$, but it may be a bit trickier to prove than Itoh's case. Le Donne and Rigot have shown that in Heisenberg groups there are also homogeneous metrics for which it fails.

\begin{lm}\label{tangraph}
Let $A\subset V\in P(n,m), g:A\to V^{\perp}$ Lipschitz and $G=G_g$. Then for $\H^m$ almost all $a\in G$, $\overline{P_V(\spt\nu)}=V$ for every $\nu\in \tanm(\H^m\restrict G,a).$
\end{lm}
\begin{proof}
Let $a\in G$ and suppose that there are $\nu\in \tanm(\H^m\restrict G,a)$ and $b\in V\setminus \overline{P_V(\spt\nu)}$, which means that there is $R>0$ for which $\nu(V\cap B(b,R)+V^{\perp})=0$. Since $\delta_r$ and $P_V$ commute, one easily checks that for $r>0$,
$$T_{a,r}^{-1}(V\cap B(b,R)+V^{\perp})=V\cap B(P_V(a)+\delta_r(b),Rr)+V^{\perp}.$$
Hence
\begin{align*}
T_{a,r\#}(\H^m\restrict G)&(V\cap B(b,R)+V^{\perp})=\H^m(G\cap (V\cap B(P_V(a)+\delta_r(b),Rr)+V^{\perp}))\\
&\geq \H^m(A\cap B(P_V(a)+\delta_r(b),Rr)).
\end{align*}
We have for some $c>0$ and some sequence $r_i>0$ with $\lim_{i\to\infty}r_i=0$, $\nu=c\lim_{i\to\infty}r_i^{-m}T_{a,r_i\#}(\H^m\restrict G)$, see \cite{Mat95}, Remark 14.4(3). It follows that 
$$\lim_{i\to\infty}r_i^{-m}\H^m(A\cap B(P_V(a)+\delta_{r_i}(b),Rr_i))=0.$$
Since $B(P_V(a)+\delta_{r_i}(b),Rr_i)\subset B(P_V(a),(R+\|b\|)r_i)$, this implies that $P_V(a)$ is not a density point of $A$. Hence by Lemma \ref{Lebdens} the $\H^m$ measure of the set of such $P_V(a)$ is zero, which is the same as to say that the set of $a\in G$ as in the first sentence of the proof has $\H^m$ measure zero.
\end{proof}
\begin{lm}\label{apprtanlm}
Let $E\subset \Pn$ be $\H^m$ measurable and $\mathcal H^m(E)<\infty$. Then at $\H^m$ almost all points of $E$ the following is true: if there is an $m$-flat measure $\lambda_a=\H^m\restrict V_a, V_a\in P(n,m),$ such that $\tanm(\H^m\restrict E,a)=\{c\lambda_a: 0<c<\infty\}$, then $V_a$ is an approximate tangent $m$-plane of $E$ at $a$. 
Conversely, for $\H^m$ almost all $a\in E$, if $\Theta_{\ast}^m(E,a)>0$ and $E$ has an approximate tangent $m$-plane $V_a$ at $a$, then $\spt\nu\subset V_a$ for every $\nu\in\tanm(\H^m\restrict E,a)$.
\end{lm}

\begin{proof}
Suppose that there is an $m$-flat measure $\lambda_a$ as in the first statement. Let $0<s<1$. By Lemma 2.5 of \cite{Mat05} the assumption implies that for $\H^m$ almost all $a\in E$ there is $c>0$ such that $\frac{1}{\H^m(E\cap B(a,r))}T_{a,r\#}(\H^m\restrict E)$ converges weakly to $c\lambda_a$ as $r\to 0$.  
As $B(0,1)\setminus X(a,V_a,s)$ is compact we have 
\begin{align*}
&0= c\lambda_a(B(0,1)\setminus X(0,V_a,s))=\\
&\lim_{r\to 0}\frac{1}{\H^m(E\cap B(a,r))}T_{a,r\#}(\H^m\restrict E)(B(0,1)\setminus X(0,V_a,s)).\\
\end{align*}
Since $\H^m(E\cap B(a,r)\setminus X(a,V_a,s))=T_{a,r\#}(\H^m\restrict E)(B(0,1)\setminus X(0,V_a,s))$ and $\limsup_{r\to 0}r^{-m}
\H^m(E\cap B(a,r))<\infty$ for $\H^m$ almost all $a\in E$ by Theorem \ref{dens}, we obtain for such $a$,
$$\lim_{r\to 0}r^{-m}\H^m(E\cap B(a,r)\setminus X(a,V_a,s))=0,$$
as required.

For the converse statement, let $a\in E$ be such that $0<\Theta_{\ast}^m(E,a)\leq\Theta^{\ast m}(E,a)<\infty$ and let $\nu\in\tanm(\H^m\restrict E,a)$. As in the proof of Lemma \ref{tangraph} we have for some $c>0$ and some sequence $r_i>0$ with $\lim_{i\to\infty}r_i=0$, $\nu=c\lim_{i\to\infty}r_i^{-m}T_{a,r_i\#}(\H^m\restrict E)$. Then the lower semicontinuity of the weak convergence yields with similar arguments as above that $\nu(U(0,R)\setminus \overline{X(0,V_a,s)})=0$ for all $R>0, 0<s<1$. Hence $\spt\nu\subset V_a$.
\end{proof}



We have the Lebesgue density theorem for parabolic rectifiable sets, recall Lemma \ref{LebH}:

\begin{thm}\label{densthm}
Let $E\subset \Pn$ be $\H^m$ measurable and LG $m$-rectifiable with $\mathcal H^m(E)<\infty$.  Then for $\mathcal H^m$ almost all $p\in E$,
$$\Theta^m(E,p)=1\ \text{or}\ \Theta^m(E,p)=2^{-m}v(m).$$
\end{thm}

\begin{proof}
If $g:A\to V^{\perp}, A\subset V\in P(n,m)$, is $L$-Lipschitz, then $x\mapsto x+g(x), x\in A,$ is $(L+1)$-bilipschitz, so it changes Hausdorff $m$ measure at most in ratio $(L+1)^m$. Hence the theorem follows from Theorem \ref{dens}(2) and Lemmas \ref{LebH} and \ref{Lebdens}, cf. the proofs of Lemma 3.2.17 and Theorem 3.2.19 in \cite{Fed69}.
\end{proof}

For horizontally rectifiable sets, see Definition \ref{horrect}, this is a special case of Kirchheim's theorem in general metric spaces, see \cite{Kir94}.

I don't know if the converse holds, or even if the analogue of Preiss's theorem \cite{Pre87} holds. That is, does the existence of positive and finite limit $\lim_{r\to 0}r^{-m}\H^m(E\cap B(p,r))$ almost everywhere imply parabolic rectifiability? In Euclidean spaces this is a very deep and difficult result, but recently Merlo succeeded in proving an analogue in Carnot groups in \cite{Mer19} and \cite{Mer20}. Maybe his methods could be adjusted to $\Pn$?


Now we complete part of the proof of Theorem \ref{main}.

\begin{proof}[Proof of (2) $\Longleftrightarrow \dots \Longleftrightarrow$ (5) in Theorem \ref{main}]
$(2) \Longrightarrow (3)$: Suppose that $E$ is LG $m$-rectifiable. Let $s_i>0, \lim_{i\to\infty}s_i=0$ and let $0<L_i<s_i$. Then for each $i$\ $E$ is $(m,L_i)$-rectifiable. By Theorem \ref{mLthm}(1) for every $i$ there is $E_i\subset E$ such that $\H^m(E\setminus E_i)=0$ and that for all $a\in E_i$ there is $V_i(a)\in P(n,m)$ for which
$$\lim_{r\to 0}r^{-m}\H^m(E\cap B(a,r)\setminus X(a,V_i(a),s_i))=0.$$
Let $a\in \cap_iE_i$ and choose a subsequence $V_{i_j}(a)$ of $V_i(a)$ converging to some $V\in P(n,m)$. If $s>0$ then for a sufficiently large $j$, 
$X(a,V_{i_j}(a),s_{i_j})\subset X(a,V,s),$
from which (3) follows.

$(3) \Longrightarrow (2)$: Suppose that $E$ has an approximate tangent $m$-plane at $\H^m$ almost all of its points. Let $L>0$ and $0<s<1$ such that $L>\tfrac{s}{\sqrt{1-s^2}}$. Then for $\H^m$ almost all $a\in E$ there is $V\in P(n,m)$ such that \eqref{M2eq} holds, whence by Theorem \ref{mLthm}(2) $E$ is $(m,L)$-rectifiable.

$(4) \Longrightarrow (3)$: This follows from Lemma \ref{apprtanlm}.

$(2)$ and  $(3) \Longrightarrow (4)$: Let $E$ be LG $m$-rectifiable. Then at $\H^m$ almost all points $a\in E$ by Lemma \ref{apprunique} $E$ has a unique approximate tangent $m$-plane $V_a$ and 
by Theorem \ref{densthm} the positive and finite density $\Theta^m(E,a)$ exists. It follows by the same proof as that of Corollary 14.9 of \cite{Mat95} in the Euclidean case that for $\H^m$ almost all $a\in E$ every $\nu\in\tanm(\H^m\restrict E,a)$ is $m$-uniform, that is, for some positive number $c$, $\nu(B(p,r))=cr^m$ for $p\in \spt\nu$ and for $r>0$. By Lemma \ref{apprtanlm} $\spt\nu\subset V_a$ and then by Lemma \ref{tangraph} $\spt\nu=V_a$. The uniqueness of uniform measures, see  Theorem 3.4 in \cite{Mat95}, implies that $\nu$ is $m$-flat.

$(4) \Longleftrightarrow (5)$: This is  Theorem \ref{M2}.

\end{proof}

Parabolic rectifiable sets split into horizontal and vertical parts.

\begin{df}\label{horrect}
A set $E\subset\R^n$ is horizontally (resp. vertically) LG \emph{$m$-rectifiable} if for every $0<L<\infty$ there are horizontal (resp. vertical) $(m,L)$-Lipschitz graphs $G_i, i=1,2,\dots,$ such that
$$\H^m(E\setminus\bigcup_{i=1}^{\infty}G_i)=0.$$
The horizontal and vertical $C^1G$ $m$-rectifiable sets are defined in the same way based on Definition \ref{Crectdf}.
\end{df}

Of course, according to this definition there are no  horizontally  LG $(n+1)$-rectifiable sets and no  vertically  LG $1$-rectifiable sets in $\Pn$. In particular, $E$ is LG $(n+1)$-rectifiable if and only if it is vertically LG $(n+1)$-rectifiable.

For subsets of $\R^{n}\times\{0\}$ the LG $m$-rectifiability is equivalent to the horizontal LG $m$-rectifiablity and to the Euclidean rectifiability, as well as for the subsets of countable unions of vertical translates of $\R^{n}\times\{0\}$.

\begin{pr}
Let $E\subset \Pn$ be $\H^m$ measurable and horizontally LG $m$-rectifiable with $\mathcal H^m(E)<\infty$. Then there is a set $T\subset\R$ such that $\mathcal L^1(T)=0$ and  $\H^m(E\setminus(\R^{n}\times T))=0$.
\end{pr}

This follows immediately from Remark \ref{graphrem}(2). We shall show in Example \ref{ex2} that we cannot always take $T$ to be countable.

The following theorem follows by the same arguments as Theorem \ref{main}, the equivalence of (1) to other conditions is still to come. 

\begin{thm}\label{main1}
Let $E\subset \Pn$ be $\H^m$ measurable and $\mathcal H^m(E)<\infty$. Then the following are equivalent:
\begin{itemize}
\item[(1)] $E$ is horizontally (resp. vertically) $C^1G$ $m$-rectifiable.
\item[(2)] $E$ is horizontally (resp. vertically) LG $m$-rectifiable.
\item[(3)] $E$ has a horizontal (resp. vertical) approximate tangent $m$-plane at $\H^m$ almost all of its points.
\item[(4)] For $\H^m$ almost all $a\in E$\ there is a horizontal (resp. vertical) $m$-flat measure $\lambda_a$ such that $\tanm(\H^m\restrict E,a)=\{c\lambda_a: 0<c<\infty\}$.
\end{itemize}
\end{thm}

Then we say that $E$ is horizontally (resp. vertically) parabolic $m$-rectifiable.

\begin{cor}
Let $E\subset \Pn$ be $\H^m$ measurable and parabolic $m$-rectifiable with $\mathcal H^m(E)<\infty$. Then there are a horizontally parabolic $m$-rectifiable set $E_1$ and a vertically parabolic $m$-rectifiable set $E_2$ such that $E=E_1\cup E_2$. The decomposition is unique up to $\H^m$ null-sets.
\end{cor}

\section{Parabolic differentiability and  completion of the proof of Theorem \ref{main}}\label{Paradiff}

To define parabolic differentiability let $H=\Rn\times\{0\}\subset\Pn$.

\begin{df}\label{paradiff}
Let $A\subset V\in P(n,m)$ and $g:A\to V^{\perp}$. We say that $g$ is \emph{parabolically differentiable} at $p\in A$ if there is a linear map 
$\Lambda_p:V\cap H\to V^{\perp}$ such that with $p=(x,t), q=(y,u)$,
\begin{equation}\label{paradiffeq}
\lim_{q\to p,q\in A,q\neq p}\|g(q)-g(p)-\Lambda_p(y-x)\|/\|q-p\|=0.\end{equation}
We say that $g$ is \emph{continuously parabolically differentiable}, denoted $C^1$, if the map $p\mapsto\Lambda_p, p\in A,$ is continuous. Then $G_g$ is called an $m-C^1$ graph over $V$.
\end{df}

For any $g$ as above we denote $\tilde{g}(p)=p+g(p)$ for $p\in A$ so that $G_g=\tilde{g}(A)$.

When $V\in H(n,m)$, this is just ordinary differentiability except for the use of the parabolic metric. When $V\in V(n,m)$, we require that the differential $\Lambda_p$ depends only on the horizontal variables, which means that the approximating planes are vertical. The graph of $\Lambda_p$, of $\Lambda_p(x,t)$ when $V$ is vertical, is the ordinary tangent plane of $G_g$ immediately by \eqref{paradiff}.

Using this notion of differentiability Orponen proved in \cite{Orp19} a Rademacher theorem for regular parabolic functions, see Section \ref{uniform} for the meaning of 'regular'. We have the following Rademacher-type theorem for parabolic rectifiable Lipschitz graphs

\begin{thm}\label{diffthm}
Let $A\subset V\in P(n,m)$ and $g:A\to V^{\perp}$. If $A$ is $\H^m$ measurable, $g$ is Lipschitz and $G=G_g$ is LG $m$-rectifiable, then $g$ is parabolically differentiable at $\H^m$ almost all $p\in A.$
\end{thm}

\begin{proof}

Let us first check the following: Let $W\in P(n,m)$ be such that $P_V:W\to V$ is bijective. Set $\Lambda=P_{V^{\perp}}\circ (P_V|W)^{-1}:V\to V^{\perp}$ and $\tilde{\Lambda}=(P_V|W)^{-1}:V\to W$, so that $W=\tilde{\Lambda}(V)$. If $h:B\to W^{\perp}, B\subset W,$ is $L$-Lipschitz and 
$\tilde{h}(q_1)=\tilde{g}(p_1)$ and $\tilde{h}(q_2)=\tilde{g}(p_2)$ for some $p_1,p_2\in A$ and $q_1,q_2\in B$, then 
\begin{equation}\label{diffeq}
\|g(p_1)-g(p_2)-\Lambda(p_1-p_2)\|\leq 4L\|\tilde{\Lambda}\|^2\|p_1-p_2\|,\end{equation}
provided $L\|\tilde{\Lambda}\|\leq 1/2$. Observe that if $V\in V(n,m)$, the values of $\Lambda$ can be written as $\Lambda(x,t)=(P_{V^{\perp}}(P_{V\cap H}|W\cap H)^{-1}(x)),0)$, so they depend only on $x$ as required in Definition \ref{paradiff}.

For $i=1,2$, set $p_i'=P_V(q_i)$. We have $p_i=P_V(\tilde{g}(p_i))=P_V(\tilde{h}(\tilde{\Lambda}(p_i')))$ and $p_i'=P_V(\tilde{\Lambda}p_i')$. Thus
\begin{align*}
&p_1-p_2= P_V(\tilde{h}(\tilde{\Lambda}(p_1')))-P_V(\tilde{h}(\tilde{\Lambda}(p_2')))\\
&=P_V(\tilde{\Lambda}(p_1'-p_2'))+ P_V(h(\tilde{\Lambda}(p_1'))-h(\tilde{\Lambda}(p_2')))=p_1'-p_2'+v,\\
\end{align*}
where $v=P_V(h(\tilde{\Lambda}(p_1'))-h(\tilde{\Lambda}(p_2')))$. 
 Hence if 
$L\|\tilde{\Lambda}\|\leq 1/2$, $\|v\|\leq L\|\tilde{\Lambda}\|\|p_1'-p_2'\|\leq \|p_1'-p_2'\|/2$ and $\|p_1'-p_2'\|\leq 2\|p_1-p_2\|$. Thus 
\begin{align*}
&\|g(p_1)-g(p_2)-\Lambda(p_1-p_2)\|=\|\tilde{g}(p_1)-\tilde{g}(p_2)-\tilde{\Lambda}(p_1-p_2))\|\\
&=\|\tilde{h}(\tilde{\Lambda}(p_1'))-\tilde{h}(\tilde{\Lambda}(p_2'))-\tilde{\Lambda}(p_1-p_2))\|\\
&=\|\tilde{\Lambda}(p_1'-p_2')+h(\tilde{\Lambda}(p_1'))-h(\tilde{\Lambda}p_2'))-\tilde{\Lambda}(p_1-p_2))\|\\
&=\|h(\tilde{\Lambda}(p_1'))-h(\tilde{\Lambda}(p_2'))-\tilde{\Lambda}(v)\|\leq L\|\tilde{\Lambda}\|\|p_1'-p_2'\|+L\|\tilde{\Lambda}\|^2\|p_1'-p_2'\|\\
&\leq L\|\tilde{\Lambda}\|(1+\|\tilde{\Lambda}\|)\|p_1'-p_2'\| \leq 2L\|\tilde{\Lambda}\|(1+\|\tilde{\Lambda}\|)\|p_1-p_2\|
\leq 4L\|\tilde{\Lambda}\|^2\|p_1-p_2\|,
\end{align*}
since $\|\tilde{\Lambda}\|\geq 1$.
The estimate \eqref{diffeq} follows from this.

Let $L$ be such that $G$ is an $(m,L)$-Lipschitz graph and choose $\e(L)>0$ such that $\tfrac{L(\e(L)+1)}{\sqrt{L^2+1}}+\e(L)<1$. Next we verify that if there is an $(m,\e(L))$-Lipschitz graph $G_h, h:B\to W^{\perp}, B\subset W\in P(n,m)$, for which $\H^m(G\cap G_h)>0$, then $P_V|W$ is bijective and
\begin{equation}\label{diffeq2}
\|(P_V)^{-1}|W\|\leq L',\end{equation}
where $L'$ depends only on $L$. Notice that by Corollary \ref{HVlm} $W$ is horizontal, if $V$ is, and vertical, if $V$ is.

Denote $\e=\e(L)$.  We may assume that $G_h\subset G$. Let $p,q\in B$. Then there are $p',q'\in A$ such that
$$p'+g(p')= p+h(p), q'+g(q')=q+h(q).$$
We have 
\begin{align*}
&\|h(p)-h(q)\|\leq \e\|p-q\|, \|g(p')-g(q')\|\leq L\|p'-q'\|,\\
&p-q=p'-q'+g(p')-g(q')-h(p)+h(q),\\
& P_{V^{\perp}}(p-q)=g(p')-g(q')-P_{V^{\perp}}(h(p)-h(q)).
\end{align*}
Hence
\begin{align*}
&\|p-q\|\geq\frac{1}{1+\e}\|p'-q'+g(p')-g(q')\|,\\
& \|P_{V^{\perp}}(p-q)\|\leq \|g(p')-g(q')\|+\e\|p-q\|,\\
&\|g(p')-g(q')\|\leq \frac{L}{\sqrt{L^2+1}}\|p'-q'+g(p')-g(q')\|.
\end{align*}
Therefore for $p,q\in B$,
\begin{equation}\label{diffeq1}
\|P_{V^{\perp}}(p-q)\|\leq\left(\frac{L(\e+1)}{\sqrt{L^2+1}}+\e\right)\|p-q\|.\end{equation}
As $B$ has positive Lebesgue measure in $W$, we can find with the help of the Lebesgue density theorem points $p_i$ in $B$ and $\lambda\in\R$ for which $\lambda(p_i-p_0)$ nearly form an orthonormal bases of $W$. This implies that \eqref{diffeq1} holds for all $p,q\in W$.  Then for $L'=(1-(\tfrac{L(\e+1)}{\sqrt{L^2+1}}+\e)^2)^{-1/2}$, we have 
$\|p\|\leq L'\|P_V(p)\|$ for $p\in W$ and our claim is proven.

Let $L_i>0$ with $\lim_{i\to\infty}L_i=0$ and $L_i<\e(L)$ where $\e(L)$ was defined above. By the LG rectifiability of $G$ there are $(m,L_i)$-Lipschitz graphs $G_{i,j}$ over closed subsets of $V_{i,j}\in P(n,m)$ such that $\H^m(G\setminus \cup_jG_{i,j})=0$ and $\H^m(G\cap G_{i,j})>0$, we may assume of course that $\H^m(G)>0$. 
Let $A_{i,j}=\tilde{g}^{-1}(G\cap G_{i,j})$. Then $\H^m(A\setminus \cup_jA_{i,j})=0$. The sets  $A_{i,j}$ are $\H^m$ measurable. Then for every $i$ and $j$ and for $\H^m$ almost all $p\in A_{i,j}$, $\Theta^{\ast m}(A\setminus A_{i,j},p)=0$. Hence for
$B_i=\cup_j\{p\in A_{i,j}: \Theta^{\ast m}(A\setminus A_{i,j},p)=0\}$
we have $\H^m(A\setminus B_i)=0$, and so $\H^m(A\setminus \cap_i B_i)=0$. 

Let $p\in \cap_i B_i$. Then there is a sequence $(A_i:=A_{i,j_i})$ such that $p\in A_i$ and $\Theta^{\ast m}(A\setminus A_i,p)=0$ for all $i$. Choose a decreasing sequence $(r_i)$ of positive numbers such that $\sum_{j>i}r_j^m<2r_{i+1}^m$ and 
$$\H^m((A\setminus A_i)\cap B(p,r))<r^m/i\ \text{for}\ 0<r\leq r_i.$$
Set 
$$B=\bigcup_{i=1}^{\infty}A_i\cap B(p,r_i)\setminus B(p,r_{i+1}).$$
If $r_{i+1}<r\leq r_i$, we have 
\begin{align*}
&\H^m((A\setminus B)\cap B(p,r))
\leq \H^m((A\setminus A_i)\cap B(p,r))+\sum_{j>i}\H^m((A\setminus A_j)\cap B(p,r_j))\\
&\leq r^m/i + \sum_{j>i}r_j^m/j \leq 3r^m/i.
\end{align*}
Hence $\Theta^{\ast m}(A\setminus B,p)=0.$

Let $V_i=V_{j_i}$. As $L_i<\e(L)$, $P_V|V_i:V_i\to V$ is bijective. Set $\Lambda_i=P_{V^{\perp}}\circ (P_V|V_i)^{-1}, \tilde{\Lambda}_i=(P_V|V_i)^{-1}$. By \eqref{diffeq2}, $\|\tilde{\Lambda}_i\|\leq L'$. Replacing $(\Lambda_i)$ by a subsequence we may assume that $\Lambda_i\to\Lambda$. For sufficiently large $i$, $L_i\|\tilde{\Lambda_i}\|\leq 1/2$  and we can apply \eqref{diffeq} to obtain
\begin{equation*}
\|g(q)-g(p)-\Lambda(q-p)\|\leq (4L_iL'^2+\|\Lambda_i-\Lambda\|)\|q-p\|)\ \text{for}\ q\in A_i.\end{equation*}
If $q\in B$, there is a unique $i$ such that $q\in A_{i}\cap B(p,r_i)\setminus B(p,r_{i+1}).$ Hence
\begin{equation*}
\|g(q)-g(p)-\Lambda(q-p)\|\leq \e(\|q-p\|)\|q-p\|\ \text{for}\ q\in B,\end{equation*}
where $\e(r)\to 0$ as $r\to 0$. This means that $g$ is parabolically approximately differentiable at $p$.  But approximate differentiability combined with Lipschitz condition implies differentiability. This is the content of Lemma 3.1.5 of \cite{Fed69} in the Euclidean setting. Exactly the same argument works in the parabolic case, so I omit it.
\end{proof}

\begin{df}\label{Crectdf}
A set $E\subset\Pn$ is called \emph{$C^1G$ $m$-rectifiable} if there are $m-C^1$ graphs $G_i$ such that $\H^m(E\setminus\cup_iG_i)=0$. 
\end{df}

We can now finish the proof of Theorem \ref{main}. We have left to show that (1) is equivalent to the other conditions.

\begin{proof}[Proof of theorem \ref{main}]
 
$(2) \Rightarrow (1)$: Let $G_{g_i}\subset E$ be $m$-Lipschitz graphs covering $\H^m$ almost all of $E$. By Theorem \ref{diffthm} each $g_i$ is parabolically differentiable almost everywhere. The map sending $p$ to the parabolic differential $\Lambda_p$ is easily seen to be Borel measurable. Hence by an application of Lusin's theorem $\H^m$ almost all of $G_{g_i}$ can be covered with countably many $m-C^1$ graphs, from which (1) follows. 

$(1) \Rightarrow (2)$: This follows from the following lemma.
\end{proof}

\begin{lm}\label{difflm}
If $A\subset V\in P(n,m)$ and $g:A\to V^{\perp}$ is parabolically differentiable in $A$, then $G_g$ is LG $m$-rectifiable.
\end{lm}

\begin{proof}
For every $p\in A$ there is $\Lambda_p:V\to V^{\perp}$ such that \eqref{paradiffeq} holds and $\Lambda_p(x,t)=\Lambda_p(y,t)$ for $x,y\in\Rn, t\in\R$. 
For each $0<L<1$ we can decompose $A$ into countably many sets $A_i$ such that for some $\Lambda_i:V\to V^{\perp}$, of the above type, 
$$\|g(q)-g(p)-\Lambda_i(q-p)\|\leq L\|q-p\|\ \text{for}\ p,q\in A_i.$$
To see this, choose first for every $p\in A$ a positive number $r(p)$ such that 
$$\|g(q)-g(p)-\Lambda_p(q-p)\|\leq L\|q-p\|/2\ \text{for}\ q\in A\cap B(p,r(p)).$$
Then $A$ is the union of the sets $B_i=\{p\in A:r(p)>1/i\}, i=1,2,\dots.$ Write each $B_i$ as the union of sets $B_{i,j}$ such that $d(B_{i,j})<1/i$.  Choose linear maps $\Lambda_k:V\cap H\to V^{\perp}, k=1,\dots,N,$ such that for every linear map $\Lambda:V\cap H\to V^{\perp}$ there is  $\Lambda_k$ for which $\|\Lambda-\Lambda_k\|<L/2$. Then $A$ is the countable union of the sets $A_{i,j,k}=\{p\in B_{i,j}: \|\Lambda_p-\Lambda_k\|<L/2\}$. This gives the required decomposition.

We shall show that each $G_{g|A_i}$ is an $(m,L(1-L^2)^{-1/2})$-Lipschitz graph, which will complete the proof.  Define 
$\tilde{\Lambda_i}(p)=p+\Lambda_i(p)$ for $p\in V$, $\tilde{h}=\tilde{g}\circ\tilde{\Lambda_i}^{-1}:\tilde{\Lambda_i}(A_i)\to G_g$ and 
$h=P_{W^{\perp}}\circ \tilde{h}:\tilde{\Lambda_i}(A_i)\to W^{\perp}$, where $W=\{\tilde{\Lambda_i}(p):p\in V\}=G_{\Lambda_i}$. Let 
$p',q'\in G_{g|A_i}, p'=\tilde{g}(p),q'=\tilde{g}(q), p,q\in A_i$. Then 
\begin{align*}
&\|P_{W^{\perp}}(p'-q')\|=\|P_{W^{\perp}}(\tilde{g}(p)-\tilde{g}(q)-\tilde{\Lambda_i}(p-q))\|\\
&=\|P_{W^{\perp}}(g(p)-g(q)-\Lambda_i(p-q))\|\leq L\|p-q\|\leq L\|p'-q'\|.
\end{align*}
Hence $G_{g|A_i}=G_h$ with $h=P_{W^{\perp}}\circ (P_W|G_{g|A_i})^{-1}$, which is $L(1-L^2)^{-1/2}$ Lipschitz.

\end{proof}

Notice that continuity of the differentials is not needed in Lemma \ref{difflm}.

\section{Euclidean rectifiability}\label{euclid}

Federer defined in 1947 a set $E\subset\R^{n}$ to be $m$-rectifiable (or countably $(\H^m,m)$ rectifiable according to his terminology) if there are Lipschitz maps $f_i:A_i\to\Rn, A_i\subset\R^m,$ such that $\H^m_E(E\setminus \cup_{i=1}^{\infty}f_i(A_i))=0$. We say then that $E$ is \emph{Euclidean $m$-rectifiable}. This is equivalent to covering with Lipschitz, or even $C^1$, graphs with small Lipschitz constants. 

In this section we answer the following question: what are the relations between parabolic and Euclidean rectifiability for sets $E$ for which both parabolic and Euclidean Hausdorff measure is positive and finite? There are two possibilities: $0<\H^{m}(E)<\infty$ and $0<\H_E^{m}(E)<\infty$ or $0<\H^{m+1}(E)<\infty$ and $0<\H_E^{m}(E)<\infty$. In the first case parabolic  rectifiability implies Euclidean by Theorem \ref{eucpr}, in the second case  Euclidean rectifiability implies parabolic by Theorem \ref{eucpr1}. The converse statements are false by Examples \ref{ex1}  and  \ref{ex3}.


We have the following very simple result:
 
\begin{thm}\label{eucpr}
If $E\subset\Pn$ is parabolic  $m$-rectifiable, then it is Euclidean $m$-rectifiable. Moreover, if $E$ is vertically parabolic  $m$-rectifiable, then the Euclidean Hausdorff dimension of $E$ is at most $m-1/2$. In particular, $\H^m_E(E)=0$. 
\end{thm}

\begin{proof}
By Lemma \ref{Hs}, $\H^m_E(A)=0$ whenever $\H^m(A)=0$. Since horizontal  $m$-Lipschitz graphs are Euclidean $m$-Lipschitz graphs,  
horizontally parabolic  $m$-rectifiable sets are Euclidean $m$-rectifiable. 

The second statement follows from Lemma \ref{HVlm1}
\end{proof}

Euclidean $m$-rectifiable sets are characterized by the almost everywhere existence of approximate tangent planes, see \cite{Mat95}, Theorem 15.19. They are defined as in Definition \ref{m-apprtan} but $\H^m$ replaced by $\H^m_E$ and the parabolic cones replaced by the Euclidean cones 
$$X_E(a,V,s) =\{p\in\Pn: d_E(p-a,V)<s|p-a|\},$$
when $V\in G(n+1,m)$, the Grassmannian of linear $m$-dimensional subspaces of $\R^{n+1}$.

In the following theorem $H=\Rn\times\{0\}$.

\begin{thm}\label{eucpr1}
If $E\subset\Pn$ is $\H_E^m$ measurable and Euclidean $m$-rectifiable, then it is parabolic  $(m+1)$-rectifiable. Moreover, if $\H_E^{m}(E)<\infty$, 
\begin{itemize}
\item[(1)] if for $\H_E^m$ almost all $a\in E$ the Euclidean approximate tangent plane of $E$ at $a$ is not contained in $H$, then $E$ is vertically parabolic  $(m+1)$-rectifiable,
\item[(2)] if for $\H_E^m$ almost all $a\in E$ the Euclidean approximate tangent plane of $E$ at $a$ is contained in $H$, then $\H^{m+1}(E)=0.$
\end{itemize} 
\end{thm}

The proof is based on the following lemma:

\begin{lm}\label{euccone}
Let $0<m<n+1$ and $V\in G(n+1,m)$ such that $V\not\subset H$. 
Let $W=\{(v,t):v\in V\cap H,t\in\R\}$. Then there is $s(V), 0<s(V)<1,$ such that for $0<s<s(V)$ there is $r(s)>0$ for which
$$X_E(a,V,s^2)\cap B(a,r(s))\subset X(a,W,s)\ \text{for all}\ a\in \Pn.$$
\end{lm}

\begin{proof}
We may assume $a=0$. Let $e=(e_1,e_2)\in V, e_2\in\R,$ with $|e|=1$ be orthogonal to $V\cap H$. Then $e_2\neq 0$. Let $p\in X_E(0,V,s^2)$. Then we can write 
$p=v+\lambda e+q$, where $v\in V\cap H, \lambda\in\R, q\in V^{\perp}$ with $q=(x,t), x\in (V\cap H)^{\perp}\subset\Rn$ and $|q| < s^2|p|.$ We need to show that when $\|p\|$ is small enough, then $|P_{W^{\perp}}(p)|<s\|p\|$, or equivalently
\begin{equation}\label{eocconeeq}
|\lambda e_1+x| < s\sqrt{|v|^2+|\lambda e_1+x|^2+|\lambda e_2+t|}.
\end{equation}
If $s<1/2$, the inequality $|q| < s^2|p|$ implies $|q| < 2s^2|v+\lambda e|$ which  means that
\begin{equation}\label{eocconeeq1}
\sqrt{|x|^2+t^2} < 2s^2\sqrt{|v|^2+\lambda^2}.
\end{equation}
Suppose that  $s<1/8$ and $|\lambda|\leq s|v|/2.$ Then $|x|\leq 4s^2|v|<s|v|/2$ and 
$$|\lambda e_1+x|\leq |\lambda|+|x|<s|v|,$$
from which \eqref{eocconeeq} follows. Suppose then that $s|v|\leq 2|\lambda|$. Then by \eqref{eocconeeq1} 
$|x|\leq 6s|\lambda|<|\lambda|$ and $|t|\leq 6s|\lambda|<|\lambda e_2|/2$ if  $s<|e_2|/12$. Hence if 
 $|\lambda|< s^2|e_2|/16$,
\begin{equation*}
|\lambda e_1+x|\leq 2|\lambda|\leq s\sqrt{|\lambda e_2|}/2\leq s\sqrt{|\lambda e_2+t|}
.\end{equation*}
 Hence \eqref{eocconeeq} follows also in this case.
\end{proof}

\begin{proof}[Proof of Theorem \ref{eucpr1}]
The first claim follows from (1) and (2).  Recall from Lemma \ref{Hs} that $\H^{m+1}\lesssim \H_E^{m}$, in particular, $\H^{m+1}(A)=0$ whenever $\H_E^{m}(A)=0$. 

Let $E$ be as assumed in (1). By Theorem 3.2.29 in \cite{Fed69} $\H^m_E$ almost all of $E$ can be covered with countably many $m$-dimensional $C^1$ submanifolds $M_i$ of $\R^{n+1}$. Each of them has a classical tangent plane $V_i(a)\in G(n+1,m)$ at all $a\in M_i$, which implies that for all $0<s<1$,
$$M_i\cap B(a,r)\setminus X_E(a,V_i(a),s^2)=\emptyset$$
for sufficiently small $r>0$. Then by the Euclidean analogue of Theorem \ref{dens},  $V_i(a)$ is a Euclidean approximate tangent plane of $E$ for $\H^m_E$ almost all $a\in E\cap M_i$, so $V_i(a)\not\subset H$ by assumption. Thus the vertical plane $W_i(a)$ related to $V_i(a)$ as in Lemma \ref{euccone} is a parabolic approximate tangent plane of $E$ for $\H^{m+1}$ almost all $a\in E\cap M_i$, again also with the help of Theorem \ref{dens}. Then (1) follows from Theorem \ref{main}.

Let $E$ be as assumed in (2). Let $\e>0$. By Theorem \ref{dens} there are $r_0>0$ and  $F\subset E$ such that   $\H^{m+1}(E\setminus F)<\e$ and
\begin{equation}\label{eucpreq}
\H^{m+1}(E\cap B)<5d(B)^{m+1}\ \text{for all parabolic balls}\ B\ \text{with}\ B\cap F\neq\emptyset\ \text{and}\ d(B)<r_0.
\end{equation}
For $\H^m_E$ almost all $a=(\a,\tau)\in E$ there are arbitrarily small $r>0$ such that
\begin{equation}\label{eucpreq1}
 H^m_E(E\cap B_E(a,r)\setminus H(a,\e,r))<\e r^m\ \text{and}\ \H^m_E(E\cap B_E(a,r))>r^m,
\end{equation}
where $H(a,\e,r)=\{(x,t):|x-\a|\leq r, |t-\tau|\leq \e r\}$. By Vitali's covering theorem for $\H^m_E$, see, , e.g., Theorem 2.8 in \cite{Mat95}), there are disjoint balls $B_E(a_i,r_i), a_i\in E,\ 0<r_i<\e,$ and $2r_i<r_0$, satisfying \eqref{eucpreq1}  such that  $\H^m_E(E\setminus \cup_iB_E(a_i,r_i))=0$. Each $B_E(a_i,r_i)\cap H(a_i,\e,r_i)$ can be covered with $N_i\lesssim \e/r_i$ parabolic balls $B_{i,j}$ of radius $r_i$. Then by \eqref{eucpreq} and \eqref{eucpreq1},
\begin{align*}
&\H^{m+1}(E)\leq\H^{m+1}(F)+\e\\
&\leq 5\sum_{i}\sum_{j=1}^{N_i}d(B_{i,j})^{m+1} + \sum_{i}\H^{m+1}(E\cap B_E(a_i,r_i)\setminus H(a_i,\e,r_i)))+\e\\
&\lesssim \sum_{i}N_ir_i^{m+1} + \sum_{i}\e r_i^m+\e \lesssim \e\sum_{i} r_i^m+\e\\ 
&\leq \e\sum_{i}\H^m_E(E\cap B_E(a_i,r_i))+\e\leq \e\H_E^m(E)+\e.
\end{align*}
Hence (2) follows.
\end{proof}

\section{Other notions of rectifiability}\label{Other}

\subsection{Uniform rectifiability}\label{uniform}

In Euclidean spaces uniformly rectifiable sets were introduced by David and Semmes in the 1990s, see \cite{DS93}. They have turned out to be more appropriate than the ordinary rectifiable sets for many topics in analysis, in particular, in harmonic analysis and partial differential equations. Although in many cases the converse is true, too. Because of this several people have been motivated to develop the corresponding theory in the parabolic setting. So far this has only been done in codimension 1, which is the most natural setting for such applications. Parabolic uniform rectifiability was defined by Hofmann, Lewis and Nystr\"om in \cite{HLN03} and \cite{HLN04}  in terms of $\beta_2$s, that is, uniform $L^2$ approximation by vertical planes. In \cite{BHHLN20} this has been shown to be equivalent to a 'big pieces of big pieces of regular vertical Lipschitz graphs' result, which means a quantitative local approximation by regular vertical Lipschitz graphs. But vertical Lipschitz graphs as defined in \ref{lipgraph} are not the right class. They need not be uniformly rectifiable. This was shown at the end of \cite{HLN03} based on an example in \cite{LS88}, see also the Observation 4.19 in \cite{BHHLN20}. Regularity of a Lipschitz graph means that the half-derivative with respect to $t$ satisfies a BMO condition. 

See \cite{BHHLN20} and \cite{BHHLN21} for some of the latest results, surveys on earlier results and further references.

For the precise definition of uniform rectifiability, see Definitions 2.20, 2.21 and 4.1 in \cite{BHHLN20} or Definition 2.19 in \cite{BHHLN21}. From these it follows that Example \ref{ex} is another example of an ordinary parabolic Lipschitz graph which is not uniformly rectifiable; use the estimates $\H^{n+1}(G\cap B(a,2r)\cap \{q:d(q-a,V)\geq sr\})\gtrsim r^{n+1}$ for $V\in V(n,n+1), a\in G, 0<r<1$.

In light of the above mentioned results on uniform rectifiability, a definition of $(n+1)$-rectifiable sets in $\Pn$ could be given using regular vertical Lipschitz graphs in place of vertical $C^1$ graphs, or vertical Lipschitz graphs with small Lipschitz constants. Mateu, Prat and Tolsa suggested that notion of rectifiability in \cite{MPT20}, page 4.  Orponen proved in Theorem 3.10 of \cite{Orp19} Rademacher's theorem  for regular Lipschitz functions from which it follows that rectifiable sets based on regular vertical Lipschitz graphs are $C^1G$ rectifiable. Example \ref{ex4} will show that the converse is false. In terms of approximation by planes, the difference between these possible concepts of rectifiability seems to be that while the parabolic rectifiability of this paper can be decided by looking at each single small scale separately the one based on regular graphs might require multiscale approximation.

Since parabolic $C^1$ and LG rectifiability can be characterized with approximate tangent planes, it would seem reasonable to expect that there could be quantitative versions of these results, that is, one type of uniform rectifiability. If so, would it agree with that of \cite{HLN03} and \cite{HLN04}? My guess is that it would not.

\subsection{Metric space rectifiability}\label{metric}
In \cite{AK00} Ambrosio and Kirchheim used the same definition as Federer to define rectifiable sets in general metric spaces $X$. They proved the characterization by approximate tangent planes in this setting, see Theorem 6.3 in \cite{AK00}. Kirchheim had already proved a substitute for Rademacher's theorem in \cite{Kir94} and using it the analogue of Theorem \ref{densthm}. There need not be any planes in $X$ but $X$ can be embedded isometrically into a Banach space, which gives a linear structure. Theorem 6.3 in \cite{AK00} is stated with a positive lower density condition, but it holds without it. To see this one can again modify the arguments of \cite{Mat95}, in particular of Lemma 15.14, or the proof of Lemma \ref{M2lemma} of this paper.

Let us say that $E\subset\Pn$ is \emph{Ambrosio-Kirchheim $m$-rectifiable}, if there are Lipschitz maps $f_i:A_i\to\Pn, A_i\subset\R^m,$ such that $\H^m(E\setminus \cup_{i=1}^{\infty}f_i(A_i))=0$.

\begin{thm}\label{AKhor}
Let $0<m< n+1$ and let $E\subset\Pn$ be $\H^m$ measurable with $\H^m(E)<\infty$. Then $E$ is horizontally parabolic  $m$-rectifiable if and only it is Ambrosio-Kirchheim $m$-rectifiable. In particular, horizontal $m$-Lipschitz graphs are horizontally parabolic  $m$-rectifiable
\end{thm}

Clearly, horizontally LG $m$-rectifiable sets and horizontal $m$-Lipschitz graphs are Ambrosio-Kirchheim rectifiable.  We shall verify the remaining part in Subsection \ref{Heisenberg}.

Now we show that codimension 1 Ambrosio-Kirchheim rectifiable sets in $\Pn$ are trivial. 

\begin{thm}\label{Lip0}
If $f:A\to\Pn, A\subset\R^{n+1},$ is Lipschitz, then $\H^{n+1}(f(A))=0$.
\end{thm}

\begin{proof}
We may assume that $A$ is bounded and since $f$ trivially has a Lipschitz extension to the closure of $A$, we may assume that $A$ is compact. Let $L$ be the Lipschitz constant of $f$,\ $N$ be a positive integer and let $a$ be a Lebesgue density point of $A$. Then for sufficiently small $R>0$, there is a closed cube $Q$ of side-length $R$ containing $a$ such that $A\cap Q_{i}\neq\emptyset$ for $i=1,\dots,N^{n+1},$ where the $Q_{i}$ form a partition of $Q$ into closed cubes of side-length $r=R/N$. Let $P, P(x,t)=t$, be the projection onto the $t$-axis. If $Q_i$ and $Q_j$ meet, then the Euclidean distance $d_E(P(f(A\cap Q_i)),P(f(A\cap Q_j)))\lesssim r^2$. The implicit constant here and below only depends on $L$ and $n$. For every $Q_i$ there is chain $Q_{i_1},\dots,Q_{i_k}, k\leq (n+1)N,$ with $Q_{i_1}=Q_1$ and $Q_{i_k}=Q_i$ such that $Q_{i_j}$ and $Q_{i_{j_+1}}$ meet, whence $d_E(P(f(A\cap Q)))\lesssim Nr^2=Rr$. Thus $f(A\cap Q)\subset Q'\times I$, where $Q'\subset\R^n$ is a cube of side-length $\lesssim R$ and $I\subset\R$ is an interval of length $\lesssim Rr$. Hence we can cover $f(A\cap Q)$ with $M, M\lesssim (R/\sqrt{Rr})^n$, parabolic balls $B_i,i=1,\dots,M$, of diameter $\sqrt{Rr}$, for which we have 
$$\sum_{i=1}^Md(B_i)^{n+1} \lesssim M(\sqrt{Rr})^{n+1}\lesssim R^n\sqrt{Rr}=R^{n+1}/\sqrt{N}.$$

We can cover $\mathcal L^{n+1}$ almost all of $A$ with cubes $Q_i$ with side-length $R_i$ satisfying the above in place of $Q$ and $R$ and such that 
$\sum_i  R_i^{n+1}\lesssim \mathcal L^{n+1}(A)$. Then each $f(A\cap Q_i)$ is covered by parabolic balls $B_{i,j},j=1,\dots,M_i$, which satisfy 
$\sum_{j=1}^{M_i}d(B_{i,j})^{n+1} \lesssim R_i^{n+1}/\sqrt{N}.$ Therefore $\mathcal H^{n+1}$ almost all of $f(A)$ is covered by the balls 
$B_{i,j},j=1,\dots,M_i,i=1,2,\dots$, for which
$$\sum_{i,j}d(B_{i,j})^{n+1} \lesssim \sum_iR_i^{n+1}/\sqrt{N}\lesssim \mathcal L^{n+1}(A)/\sqrt{N}.$$
As we can take $N$ arbitrarily large, it follows that $\H^{n+1}(f(A))=0$.
\end{proof}

\subsection{Heisenberg groups}\label{Heisenberg}
Heisenberg group $\h^1$ is somewhat similar to $\p^2$, but with essential differences. It is $\R^{3}$ with a non-Abelian group structure and non-Euclidean metric. Balls have the same ellipsoidal shape as in $\p^2$ but the structure is more complicated in that the axis of the ellipsoids turn when the center moves. On the other hand, $\h^1$ also has advantages over $\p^2$. It is a geodesic space; any two points can be joined with a rectifiable curve of minimal length, whereas in $\p^2$ no two points with different $t$ coordinates can be joined by any rectifiable curve. There are also many analytical tools available in $\h^1$, for example Pansu's Rademacher theorem. The same comments apply to $\hn=\R^{2n+1}, n\geq 2,$ but there are more differences to parabolic spaces.

Franchi, Serapioni and Serra Cassano \cite{FSS01} used two different definitions for rectifiability in $\hn$. For low-dimensional, $m\leq n$, sets Federer's definition is fine, but for high-dimensional sets, $m\geq n+2$, it would only give trivial zero measure sets, recall from Theorem \ref{Lip0} the same in $\Pn$ for codimension 1. For $m$-dimensional, $m\geq n+2$, sets \cite{FSS01} used covering with $C^1$ (in the Heisenberg sense) level sets.

An analogue of Theorem \ref{main} in Heisenberg groups was proved in \cite{MSS10}, and in \cite{IMM20} for horizontal sets in general homogeneous groups. Antonelli and Merlo have developed in \cite{AM20} and \cite{AM21} rectifiability theory  in arbitrary Carnot groups with many deep results, in particular on relations between rectifiability, tangent planes and tangent measures including an analogue of Theorem \ref{main}.

In Heisenberg and Carnot groups intrinsic differentiable and  Lipschitz graphs have recently played a fundamental role in many respects. They are defined geometrically in terms of cones. In Heisenberg groups a Rademacher theorem was proved in \cite{FSS11} for codimension 1 sets and recently by Vittone for vertical sets of general dimensions in \cite{Vit20}. But Julia, Nicolussi Golo and Vittone showed in \cite{JNV21} that this is false in some Carnot groups.

For $n\geq 1$\ $\Pn$ can be considered as a vertical subspace of $\h^{n}=\R^{2n+1}$ identifying $(x,t)\in\Pn$ with $(x,0,t)\in \h^{n}$. In this way Chousionis and Tyson \cite{CT15} obtained Marstrand's theorem 'existence of $s$-density implies that $s$ is an integer' in $\Pn$ after proving it in  Heisenberg groups. In the same manner we can finish the proof of Theorem \ref{AKhor}: For $m\leq n$ Ambrosio-Kirchheim and Franchi-Serapioni-Serra Cassano $m$-rectifiable sets are the same immediately by the definitions. For subsets of $\Pn\subset\h^{n}$ the condition (3) of Theorem \ref{main1} characterizing horizontally parabolic  $m$-rectifiable sets is exactly the same as the characterization \cite{MSS10}, Theorem 3.14(iii), of Franchi-Serapioni-Serra Cassano $m$-rectifiable sets. Hence these two classes are the same, too.

I suppose it is not necessary to go via Heisenberg groups to prove this, but since it was available, I didn't try to find another way.

\section{Examples}\label{example}
We now construct several examples. 
We first show that Lipschitz graphs need not be parabolic  rectifiable.

\begin{ex}\label{ex}
There are a function $f:[0,1]\to\R$ such that $f$ is H\"older continuous with exponent $1/2$ and positive numbers $c$ and $s, 0<s<1,$ with the following properties: Let $G=\{(f(t),y,t)\in\Pn:y\in [0,1]^{n-1},t\in [0,1]\}$ be the graph of the Lipschitz function $g, g(0,y,t)=(f(t),0,0)$. Then for all $V\in V(n,n+1), a\in G$ and $0<r<1$ there is $p\in G\cap B(a,r)$ for which $B(p,sr)\subset \{q:d(q-a,V)\geq sr\}$ and $\H^{n+1}(G\cap B(p,sr))\geq cr^{n+1}$. Hence $G$ does not have any approximate tangent planes and so it is not parabolic  $(n+1)$-rectifiable. 
\end{ex}

\begin{proof}
The last inequality follows from the AD-regularity of Lipschitz graphs. Let $f$ be the Weierstrass function defined by 
$$f(t)=c_0\sum_{k=1}^{\infty}2^{-k/2}\cos (2^kt),\ t\in [0,1].$$
Then if $c_0>0$ is sufficiently small, $f$ has the following properties:
\begin{equation}\label{exeq1}
|f(t)-f(t')|\leq \sqrt{|t-t'|}\ \text{for}\ t,t'\in [0,1],\end{equation}
there is $c>0$ such that for every subinterval $I$ of $[0,1]$ there are $t,t'\in I$ such that
\begin{equation}\label{exeq2}
|f(t)-f(t')|> c\sqrt{d(I)}.\end{equation}
For the proofs, see \cite{BP17}, Lemma 5.1.8 and Theorem 5.3.1.

Let $0<s<1/4$, $0<r<1$, $a=(f(\tau),\beta,\tau)\in G$ and $V\in V(n,n+1)$. Then for some unit vector $e=(e_1,e_2)\in\R^n, e_1\in\R, e_2\in\R^{n-1}$,\ $V^{\perp}=\{(\lambda e,0):\lambda\in\R\}$ and for $p=(f(t),y,t)\in G$,
$$P_{V^{\perp}}(p)=((e_1f(t)+e_2\cdot y)e,0),$$
whence 
$$d(p-a,V)= |P_{V^{\perp}}(p-a)|=|e_1(f(t)-f(\tau))+e_2\cdot (y-\beta)|.$$
Suppose that $|e_1|\geq 1/2$ and let $p=(f(t),\beta,t)$. If $d(p-a,V)<2sr$, then  $|f(t)-f(\tau)|<2sr/|e_1|\leq 4sr$.
Suppose that $\tau + r^2/2\leq 1$, the other case $\tau - r^2/2\geq 0$ can be dealt 
with in the same way.  By \eqref{exeq2}, there are $t,t'\in [\tau+r^2/4,\tau + r^2/2]$ such that $|f(t)-f(t')|\geq cr/2$. Consider 
$p=(f(t),\beta,t),p'=(f(t'),\beta,t')\in G$. Then $p,p'\in B(a,r)$. If $d(p-a,V)<2sr$ and $d(p'-a,V)<2sr$, then by the above
$$cr/2\leq |f(t)-f(t')|< 8sr.$$
This is impossible if we choose $s=c/16$. With this choice $d(p-a,V)\geq 2sr$ or $d(p'-a,V)\geq2sr$. Suppose this holds for $p$. Then 
$B(p,sr)\subset  \{q:d(p-a,V)\geq sr\}$, as required.

Suppose then that $|e_1|<1/2$, so $|e_2|>1/2$. Let $0<r<1/2$ and choose $p=(f(\tau),y,\tau)\in G$ such that $y-\beta$ is parallel to $e_2$ and $r=|y-\beta|=\|p-a\|$. Then 
$$d(p-a,V)=|P_{V^{\perp}}(p-a)|=|e_2|r> r/2\geq 2sr,$$
since $s<1/4$. This  completes the proof.
\end{proof} 

The following example shows that parabolic $(n+1)$-rectifiable sets in $\Pn$ need not be rectifiable in terms of the regular Lipschitz graphs, recall the discussion in Subsection \ref{uniform}. I don't give a formal definition of these graphs, I only state a weaker property that will suffice: if $G_g\subset\Pn$ is a regular $(n+1)$-Lipschitz graph over $V\times\R, V\in G(n,n-1)$, then there are a measurable function $\phi:C\to\R, C\subset \R^{n-1}\times\R$,  and a rotation $\rho$ of $\R^{n}$ such that $\rho(\{0\}\times\R^{n-1})=V$ and $g(\rho(0,y),t)=\phi(y,t)e_V$, where $e_V\in\R^{n-1}$ is a unit vector orthogonal to $V$ and,
\begin{equation}\label{ex4eq}
\int_{u:(y,u)\in C}\frac{|\phi(y,t)-\phi(y,u)|^2}{|t-u|^2}\,du<\infty\ \text{for almost all}\ (y,t)\in C,
\end{equation}
see \cite{DDH18}, Theorem 2.3.

\begin{ex}\label{ex4}
There is an $\H^{n+1}$ measurable parabolic $(n+1)$-rectifiable set $E\subset\Pn$ such that $0<\H^{n+1}(E)<\infty$ and $\H^{n+1}(E\cap G)=0$ for every regular $(n+1)$-Lipschitz graph $G$.
\end{ex}

\begin{proof}
Let $f_0:[0,1]\to\R$ be a function satisfying \eqref{exeq1} and  \eqref{exeq2}, $f_0$ and $c$ will be fixed throughout the proof.

Let us first check the elementary fact that for every subinterval $I$ of $[0,1]$ there are   
\begin{equation}\label{ex4eq2}
a,b\in I, a<b,\ \text{such that}\ f_0(a)=f_0(b)\ \text{and}\ b-a\geq \tilde{c}d(I)\ \text{where}\ \tilde{c}=c^4/128.  
\end{equation}
To prove this let $N$ be an integer such that $1/N<c^2/16\leq 2/N$ and let $I_j, j=1,\dots,2N,$ be subintervals of $I$ of length $d(I)/(4N)$ and 
$d(I_j,I_{j+1})=d(I)/(4N)$ for $j=1,\dots,2N-1$. Let $t_j<t_j'$ in $I_j$ be the points given by \eqref{exeq2}. We may assume that there
are $N$ intervals $I_{j_i}, i=1,\dots,N$, such that for $u_i=t_{j_i}<u_i'=t_{j_i}'<u_{i+1}=t_{j_{i+1}}, i\leq N-1,$ we have $f_0(u_i')-f_0(u_i)> c\sqrt{d(I)/(4N)}$. Notice that also $|u_i-u_i'|> c^2d(I)/(4N)$. If 
 $f_0(u_{i+1})\leq f_0(u_i')$, then there are $a$ and $b$ such that $f_0(a)=f_0(b)$ and either $a\leq u_i'$ and $b= u_{i+1}$ or $a= u_i$ and $b\geq u_{i}'$. In both cases $b-a\geq c^2d(I)/(4N)\geq (c^4/128)d(I)$. We assume then  that $f_0(u_i')< f_0(u_{i+1})$ for all $i=1,\dots,N-1$ and derive a contradiction. Then
\begin{align*}
&\sqrt{d(I)} \geq f_0(u_N')-f_0(u_1') =\sum_{i=1}^{N-1}(f_0(u_{i+1}')-f_0(u_{i+1})+f_0(u_{i+1})-f_0(u_i'))\\
&\geq \sum_{i=1}^{N-1}(f_0(u_{i+1}')-f_0(u_{i+1}))\geq (N-1)c\sqrt{d(I)/(4N)}\geq c\sqrt{Nd(I)}/4>\sqrt{d(I)}.
\end{align*}

Let $I$ be a subinterval of $[0,1]$ and let $f:I\to\R, 0<L'<L\leq 1, 0<c''<c', t_I,t_I'\in I,$ be such that 
\begin{equation}\label{exeq3}
|f(t)-f(t')|\leq L\sqrt{|t-t'|}\ \text{for}\ t,t'\in I,\end{equation}
\begin{equation}\label{exeq4}
|f(t_I)-f(t_I')|> c'L\sqrt{d(I)}.\end{equation}
Now we shall show that there are non-overlapping closed intervals $I_j=[a_j,b_j]\subset I, j=1,2,\dots,$ such that $f(a_j)=f(b_j)$ for all $j$ and letting $A_I=\cup_jI_j$ we have $\mathcal L^1(I\setminus A_I)=0$ and there is a continuous $f_I:I\to\R$ for which
\begin{equation}\label{exeq4a}
f_I(t)=(L'/L)(f(t)-f(a_{j}))+f(a_{j})\ \text{for}\ t\in I_{j},\end{equation}
\begin{equation}\label{exeq5}
|f_I(t)-f_I(t')|\leq L\sqrt{|t-t'|}\ \text{for}\ t,t'\in I,\end{equation}
\begin{equation}\label{exeq6}
|f_I(t)-f_I(t')|\leq L'\sqrt{|t-t'|}\ \text{for}\ t,t'\in I_j, j=1,2,\dots,\end{equation}
\begin{equation}\label{exeq7}
|f_I(t_I)-f_I(t_I')|> c''L'\sqrt{d(I)}.\end{equation}

Let $N$ be a positive integer with $0< 1/\sqrt{N} < (c'L-c''L')/2.$ Partition $I$ into $N$ intervals $I_1',\dots,I_N'$ of length $d(I)/N$. Choose first by \eqref{ex4eq2} $a_{1,1},b_{1,1}\in I_1', a_{1,1}<b_{1,1},$ such that $f(a_{1,1})=f(b_{1,1})$ and $b_{1,1}-a_{1,1} \geq \tilde{c}d(I_1')$ and set $I_{1,1}=[a_{1,1},b_{1,1}]$. Let $J_1, J_2 $ be the complementary intervals of $I_{1,1}; I_1'\setminus (a_{1,1},b_{1,1})=J_1\cup J_2$ (with only one of them if $a_{1,1}$ or $b_{1,1}$ is an end-point of $I_1'$). Let $a_{1,2},b_{1,2}\in J_1, a_{1,2}<b_{1,2},$ such that $f(a_{1,2})=f(b_{1,2})$ and $b_{1,2}-a_{1,2} \geq \tilde{c}d(J_1)$. Set $I_{1,2}=[a_{1,2},b_{1,2}]$. In the same way find $I_{1,3}=[a_{1,3},b_{1,3}]\subset J_2$. Then $\mathcal L^1(I_1'\setminus I_{1,1})\leq (1-\tilde{c})d(I)$ and
$$\mathcal L^1(I_1'\setminus (I_{1,1}\cup I_{1,2} \cup I_{1,3}))\leq (1-\tilde{c})^2d(I);$$ 
the measure of the complement of the selected intervals decreases geometrically. Continuing in this manner we find the intervals $I_{1,j}=[a_{1,j},b_{1,j}]\subset I_1', j=1,2,\dots,$ such that $f(a_{1,j})=f(b_{1,j})$ and $\mathcal L^1(I_1'\setminus \cup_jI_{1,j})=0.$ 

Next we perform the same operation on each $I_2',\dots,I_N'$ getting the intervals $I_{k,j}$ for $k=1,\dots,N$. We denote all these by $I_j=[a_j,b_j]$ and set $A_I=\cup_jI_j$. Then
$\mathcal L^1(I\setminus A_I)=0$ and $f(a_j)=f(b_j)$ for all $j$. Define $f_I$ on  $A_I$ by $f_I(t)=(L'/L)(f(t)-f(a_{j}))+f(a_{j})$ for $t\in I_{j}$ and set $f_I(t)=f(t)$ when $t\in I\setminus A_I$. Then $f_I$ is continuous, since it agrees with $f$ at the end-points of every $I_j$. 

We now check the properties \eqref{exeq4a} - \eqref{exeq7}. First \eqref{exeq4a} holds by definition. To check \eqref{exeq5} we may assume that for some $j,j', t\in I_j$ and $ t'\in I_{j'}$, since $A_I$ is dense. Suppose $j\neq j'$  and $t<t'$. Then $t\leq b_j\leq a_{j'}\leq t'$ and we get by \eqref{exeq3}
\begin{align*}
&|f_I(t)-f_I(t')|=|(L'/L)(f(t)-f(t'))+(1-(L'/L))(f(b_{j})-f(a_{j'}))|\\
&\leq L'\sqrt{|t-t'|}+(L-L')\sqrt{|b_j-a_{j'}|}\leq L\sqrt{|t-t'|}.
\end{align*}
The case $j=j'$ is  easier and gives \eqref{exeq6}.

Let us first check \eqref{exeq7} when $t_I\in I_j$ and $ t_I'\in I_{j'}$ for some $j,j'$. Then by \eqref{exeq4},
\begin{align*}
&|f_I(t_I)-f_I(t_I')|\geq |f(t_I)-f(t_I')|-(1-L'/L)(|f(t_I)-f(a_j)|+|f(t_I')-f(a_{j'})|)\\
&>c'L\sqrt{d(I)}-2\sqrt{d(I)/N}>c''L'\sqrt{d(I)}\end{align*}
by the choice of $N$. If, for example, $t_I$ is not in $A_I$, then there are $a_{j_i}$ converging to $t_I$ and $f_I(a_{j_i})=f(a_{j_i})$ from which \eqref{exeq7} follows by a similar argument. 

Let $(L_k)$ and $(c_k)$ be strictly decreasing sequences of positive numbers with $L_1<1$ and $c/2<c_k<c$. Now we perform a recursive construction by first applying the above with $f=f_0, I=[0,1], L=1, L'=L_1, c'=c, c''=c_1, t_I=t$ and $t_I'=t'$, where $t,t'\in [0,1]$ are as in \eqref{exeq2}. Set $A_1=A_I, f_1=f_I$ and 
$I_{1,j}=[a_{1,j},b_{1,j}]=[a_{j},b_{j}]$. Suppose that for some $k\geq 1$ we have constructed the non-overlapping intervals $I_{k,j}=[a_{k,j},b_{k,j}]$,   
the continuous function $f_k:[0,1]\to\R$ and the points $t_{k,j},t_{k,j}'\in I_{k,j},$ such that
\begin{equation}\label{exeq9}
f_{k}(t)=(L_{k}/L_{k-1})(f_{k-1}(t)-f_{k-1}(a_{k,j}))+f_{k-1}(a_{k,j})\ \text{for}\ t\in I_{k,j},\end{equation}
\begin{equation}\label{exeq10}
|f_k(t)-f_k(t')|\leq L_l\sqrt{|t-t'|}\ \text{for}\ t,t'\in I_{l,j},l\leq k, j=1,2,\dots,\end{equation}
\begin{equation}\label{exeq11}
|f_k(t_{l,j})-f_k(t_{l,j}')|> c_kL_l\sqrt{d(I_{l,j})}.\end{equation}

Then we apply the above construction on each $I_{k,j}$ with $f=f_k,L=L_k, L'=L_{k+1}, c'=c_k, c''=c_{k+1},$  to obtain $I_{k+1,j}=[a_{k+1,j},b_{k+1,j}]$ with $f_{k}(a_{k+1,j})=f_{k}(b_{k+1,j})$ and $A_{k+1}=\cup_jI_{k+1,j}$ with $\mathcal L^1([0,1]\setminus A_{k+1})=0$. Again each $I_{k,j}$ is first divided into $N_k$ equal parts where $0< 1/\sqrt{N_k} < (c_kL_k-c_{k+1}L_{k+1})/2.$ We define $f_{k+1}$ on $A_{k+1}$ by
\begin{equation}\label{exeq12}
f_{k+1}(t)=(L_{k+1}/L_{k})(f_{k}(t)-f_{k}(a_{k+1,j}))+f_{k}(a_{k+1,j})\ \text{for}\ t\in I_{k+1,j}\end{equation}
and extend it to $[0,1]$ by continuity. Then by the construction,
\begin{equation}\label{exeq13}
|f_{k+1}(t)-f_{k+1}(t')|\leq L_l\sqrt{|t-t'|}\ \text{for}\ t,t'\in I_{l,j},l\leq k+1, j=1,2,\dots.\end{equation}
Finally we have
\begin{equation}\label{exeq14}
|f_{k+1}(t_{l,j})-f_{k+1}(t_{l,j}')|> c_{k+1}L_l\sqrt{d(I_{l,j})}\ \text{for}\ l\leq k+1, j=1,2,\dots\end{equation}
where the points $t_{k+1,j},t_{k+1,j}'\in I_{k+1,j},$ still need to be introduced. 
Going down from $k$ to $0$ with the formula \eqref{exeq9} we see that $f_{k+1}|I_{k+1,j}=L_{k+1}f_{0}|I_{k+1,j}+\a_{k+1,j}$
for some $\a_{k+1,j}\in\R$. Hence we can apply  \eqref{exeq2} to find $t_{k+1,j}$ and $t_{k+1,j}'$ so that \eqref{exeq14} holds for $l=k+1$ even with $c$ in place of $c_{k+1}$. For $l<k+1$ \eqref{exeq14} follows from \eqref{exeq11} by the same argument we used for \eqref{exeq7}.

Note that by the construction each interval of level $k+1$ is contained in a unique $k$ level interval. We choose the integers $N_k$ large enough so that 
\begin{equation}\label{exeq2a}
d(I_{k+1,i}) \leq (c/16)^2d(I_{k,j})\ \text{whenever}\ I_{k+1,i}\subset I_{k,j}.
\end{equation}

Let $A=\cap_kA_k$. Then $\mathcal L^1([0,1]\setminus A)=0$. For every $t\in A$ the sequence $f_k(t)$ converges, because, if $t\in I_{k,j}$, then by \eqref{exeq2a},
$$|f_k(t)-f_{k-1}(t)|=(1-L_k/L_{k-1})|f_{k-1}(t)-f_{k-1}(a_{k,j})|\leq \sqrt{d(I_{k,j})}\leq 2^{-k}.$$
Let $f=\lim_{k\to\infty}f_k:A\to\R$. By \eqref{exeq13}  and \eqref{exeq14} $f$ has the following properties for all $j,k$. 

\begin{equation}\label{exeq15}
|f(t)-f(t')|\leq L_k\sqrt{|t-t'|}\ \text{for}\ t,t'\in I_{k.j},\end{equation}
\begin{equation}\label{exeq16}
|f(t_{k,j})-f(t_{k,j}')|> L_k (c/2)\sqrt{d(I_{k,j})}.\end{equation}

We choose the sequence $(L_k)$ so that 
 \begin{equation}\label{exeq17}
\lim_{k\to\infty}L_k=0\ \text{and}\ \sum_kL_k^2=\infty.\end{equation}

Next we verify that if $B\subset A$ is Lebesgue measurable, then

\begin{equation}\label{exeq18}
\int_B\frac{|f(t)-f(u)|^2}{|t-u|^2}\,du=\infty\ \text{for almost all}\ t\in B.\end{equation}

We shall prove this at every density point $t$ of $B$. There are intervals $I_k:=I_{k,j_k}, k=1,2,\dots,$ such that $t\in I_k$ for all $k$. Let $d_k=d(I_{k,j_k}),t_k=t_{k,j_k},t_k'=t_{k,j_k}'$. Then by \eqref{exeq16} and \eqref{exeq15}, $|f(t_k)-f(t_{k}')|> L_k (c/2)\sqrt{d_k}$ and $|t_k-t_k'|>(c/2)^2d_k$. If, for example, $|f(t)-f(t_{k})|> L_k (c/4)\sqrt{d_{k}}$, then by \eqref{exeq15} 
$|f(t)-f(u)|> L_k (c/8)\sqrt{d_{k}}$ when $u\in I_k$ and $|u-t_k|\leq (c/8)^2d_k$. Hence  there is an interval $J_k\subset I_k$ such that $d(J_k)= (c/8)^2d_k$ and $|f(t)-f(u)|> L_k (c/8)\sqrt{d_{k}}$ for $u\in J_k$. Since  $\mathcal L^1(I_{k+1})\leq (c/16)^2d_k$ by \eqref{exeq2a} and $t$ is a density point of $B$, for sufficiently large $k$ we have $\mathcal L^1(B\cap J_k\setminus I_{k+1})\geq(c/32)^2d_k$, which implies
$$\int_{B\cap I_k\setminus I_{k+1}}\frac{|f(t)-f(u)|^2}{|t-u|^2}\,du\geq\frac{(L_k (c/8)\sqrt{d_{k}})^2}{((c/8)^2d_k)^2}(c/32)^2d_k= L_k^2/16,$$
and so by \eqref{exeq17}
$$\int_B\frac{|f(t)-f(u)|^2}{|t-u|^2}\,du \geq \sum_kL_k^2/16=\infty.$$

Let $E=\{(f(t),y,t):y\in [0,1]^{n-1}, t\in A\}=G_h$ with $h(y,t)=(f(t),0,0)$. Then $E$ is a Borel set with $0<\H^{n+1}(E)<\infty$. For every $k$, $E$ is contained in the union of the $(n+1,L_k)$-Lipschitz graphs $G_{h|[0,1]^{n-1}\times I_{k,j}}, j=1,2,\dots,$ so $E$ is LG $(n+1)$-rectifiable.

Now we can finish the proof. Let $G=G_g\subset\Pn$ be a regular Lipschitz $(n+1)$-graph over $V\times\R, V\in G(n,n-1)$, as in \eqref{ex4eq} and before it:
$$G_g=\{(\rho(0,y)+\phi(y,t)e_V,t):(y,t)\in C\}.$$
Let $e_1=(1,0,\dots,0)\in\Rn$. If $(f(t),y,t)\in E\cap G$, then for some $y'\in \R^{n-1}, (f(t),y,t)=(\rho(0,y')+\phi(y',t)e_V,t)$ and
$$f(t)=e_1\cdot (\rho(0,y')+\phi(y',t)e_V).$$
For $y\in [0,1]^{n-1}$ let $E_y=\{t\in [0,1]:(f(t),y,t)\in E\cap G\}.$
Then by \eqref{ex4eq} for almost all $y\in [0,1]^{n-1}$ and almost all $t\in E_y$,
$$\int_{E_y}\frac{|f(t)-f(u)|^2}{|t-u|^2}\,du<\infty.$$
On the other hand, whenever $\mathcal L^1(E_y)>0$ this integral is infinite for almost all $t\in E_y$ by \eqref{exeq18}. Hence $\mathcal L^1(E_y)=0$ for almost all $y\in [0,1]^{n-1}$ and so, with $P(x,y,t)=(y,t)$, $\H^{n+1}(P(E\cap G))=c(n)\mathcal L^{n}(P(E\cap G))=0$ by Fubini's theorem. Hence $\H^{n+1}(E\cap G)=0$, because $E\cap G$ is a countable union of parabolic Lipschitz graphs over subsets of $P(E\cap G)$.
\end{proof}

The vertical axis is Euclidean 1-rectifiable but not parabolic  1- rectifiable, but it has infinite $\H^1$ measure. We now show that neither the Euclidean rectifiable sets with finite parabolic Hausdorff measure need to be parabolic  rectifiable. 

\begin{ex}\label{ex1}
For $0<m\leq n$ there is a compact Euclidean $m$-rectifiable set $E\subset\Pn$ with $0<\H_E^m(E)<\infty$ and $0<\H^m(E)<\infty$ which is not parabolic  $m$-rectifiable. \end{ex}

\begin{proof}
I first perform the construction in $\p^1$. 
Let $(n_j)$ be a strictly increasing sequence of positive integers and set $N_k=n_1\cdot\cdot\cdot n_k$ and $r_k=1/N_k$ for $k=1,2,\dots.$ For each $k$ partition $[0,1]$ into intervals $I_{k,i}=[x_{k,i},x_{k,i}+r_k], i=1,\dots,N_k,$ of length $r_k$. Thus every $I_{k,i}$ is partitioned into $n_{k+1}$ intervals $I_{k+1,j}$. Define the line segments 
$$J_{k,i}=\{(x,t_{k,i}+r_k(x-x_{k,i})):x\in I_{k,i}\}, i=1,\dots,N_k, k=1,2,\dots,$$ 
so each $J_{k,i}$ has length $r_k\sqrt{1+r_k^2}$ and slope $r_k$. The $t_{k,i}$ are chosen so that $t_{1,i}=0$ and for $k\geq 2$ the left end-point of $J_{k,i}$ lies on $J_{k-1,j}$, where $I_{k,i}\subset I_{k-1,j}$.

Let $R_{k,i}=I_{k,i}\times [t_{k,i},t_{k,i}+r_k^2].$ Then each $R_{k+1,i}$ is contained in some $R_{k,j}$ and the parabolic diameter 
$d(R_{k,i})\leq 2r_k$. Define
$$E=\bigcap_{k=1}^{\infty}\bigcup_{i=1}^{N_k}R_{k,i}.$$
Then $1\leq\H_E^1(E)\leq 2$ and $1\leq\H^1(E)\leq 2$. For any $p\in E$ and $0<r<1$, the projection of $E\cap B(p,r)$ on the $x$-axis contains an interval of length $r$, which implies 
\begin{equation}\label{ex1eq}
\H^1(E\cap B(p,r))\geq r.\end{equation}

If $a=(\a,\tau)\in E$, then for a sequence $i(k), \a\in I_{k,i(k)}$,  
and there are $\tau_k$ for which $(\a,\tau_k)\in J_{k,i(k)}$ and $0\leq\tau_k-\tau \leq r_{k+1}r_k$, where  $r_{k+1}r_k/r_k^2=r_{k+1}/r_k\to 0$ as $k\to \infty$. 
Let $p_k=(x_k,t_k)\in E$ with $x_k\in I_{k,i(k)}$. Then there is  $t_k'$ such that  $(x_k,t_k')\in J_{k,i(k)}$ and $0\leq t_k'-t_k \leq r_{k+1}r_k$. Now
\begin{align*}
&T_{a,r_k}(p_k)=\left(\frac{x_k-\a}{r_k},\frac{t_k-\tau}{r_k^2}\right)\\
&=\left(0,\frac{t_k-t_k'}{r_k^2}\right)+\left(\frac{x_k-\a}{r_k},\frac{t_k'-\tau_k}{r_k^2}\right)+\left(0,\frac{\tau_k-\tau}{r_k^2.}\right)\end{align*}
The second coordinates of the first and third term in the sum tend to zero when $k\to\infty$. The middle term is
$$\left(\frac{x_k-\a}{r_k},\frac{t_k'-\tau_k}{r_k^2}\right)=\left(\frac{x_k-\a}{r_k},\frac{x_k-\a}{r_k}\right).$$
Suppose, for example, that there are infinitely many $k_1,k_2,\dots,$ such that $\a$ belongs to the left half of $I_{k_i,i(k_i)}$. Then $(x_{k_i}-\a)/r_{k_i}$ ranges between $0$ and $1/2$ when $x_{k_i}$ ranges between $\a$ and $\a+r_{k_i}/2$. It follows that $T_{a,r_{k_i}}(E)$ converges to a closed set which contains the segment $\{(y,y):0 \leq y\leq 1/2\}$. Using \eqref{ex1eq} this implies that the horizontal line cannot be an approximate tangent line of $E$ at $a$. Neither can the vertical line be, but this is not relevant here, since $E$ is 1-dimensional. But for higher dimensional modifications of the construction this should be taken into account.

Finally, to show that $E$ is Euclidean 1-rectifiable, connect first the end-points of the segments $J_{1,i}$ to each other by vertical segments of length $r_1^2$ and then for each $k\geq 2$ connect the end-points of the segments $J_{k,i}$ to each other by vertical segments of length $r_kr_{k-1}-r_{k}^2$. This gives curves $\Gamma_k$ of length at most 3. The limit of them is a rectifiable curve containing $E$. This last argument does not work for higher dimensional examples, but one can also easily show that the horizontal line is a Euclidean approximate tangent line and this can be generalized.

For general $0<m\leq n$, let $E\subset\p^1$ be as above. Then $\{0\}^{n-m}\times [0,1]^{m-1}\times E$ works. I leave the routine checking to the reader.
\end{proof}

I only do the following example in $\p^1$, but similar examples can be given in higher dimensions.

\begin{ex}\label{ex3}
There is a compact vertically parabolic $2$-rectifiable set $E\subset\p^1$ with $0<\H^{2}(E)<\infty$ and $0<\H_E^1(E)<\infty$ which is not Euclidean  $1$-rectifiable. \end{ex}

\begin{proof}
Let $r_k>0, k=1,2,\dots,$ be such that $r_1=1, \lim_{k\to\infty} r_k=0$ and $r_k=n_kr_{k+1}$ for some strictly increasing sequence of positive even integers $n_k$. Let $N_0=1$ and $N_k=n_1\cdot\cdot\cdot n_k$ for $k\geq 1$. Then $r_{k+1}=1/N_k$. We construct squares $Q_{k,i}, i=1,\dots,N_{k-1},$ with side-length $r_k$ and rectangles $R_{k,i}, i=1,\dots,2N_{k-1},$ with side-lengths $r_{k+1}$ and $r_k/2$: starting with $Q_{1,1}=[0,1]\times [0,1], R_{1,1}=[0,r_2]\times [0,1/2], R_{1,2}=[1-r_2,1]\times [1/2,1]$. We define recursively;
\begin{align*}
&\text{if}\ Q_{k,i}=[a_{k,i},a_{k,i}+r_k]\times [b_{k,i},b_{k,i}+r_k],\text{then}\\\ 
&R_{k,2i-1}=[a_{k,i},a_{k,i}+r_{k+1}]\times [b_{k,i},b_{k,i}+r_{k}/2], \\
&R_{k,2i}=[a_{k,i}+r_k-r_{k+1},a_{k,i}+r_k]\times [b_{k,i}+r_{k}/2,b_{k,i}+r_{k}].\\
\end{align*}
For $k>1$ the squares $Q_{k,i}$ are obtained by partitioning each $R_{k-1,i}$ into $n_{k-1}/2$ squares of side-length $r_k$.  Define  
$$E=\bigcap_{k=1}^{\infty}\bigcup_{i=1}^{N_{k-1}}Q_{k,i}=\bigcap_{k=1}^{\infty}\bigcup_{i=1}^{2N_{k-1}}R_{k,i}.$$

If $a\in R_{k,i}$, there is $j\neq i$ such that $R_{k,i}\cup R_{k,j}\subset B_E(a,2r_{k})$. From this one easily sees that $E$ cannot have Euclidean approximate tangent lines at any of its points, so it is Euclidean purely 1-unrectifiable.

The projection of $E$ on the vertical axis is $[0,1]$ which implies that both $\H^{2}(E)$ and $\H_E^1(E)$ are positive. For every $k$ we have $\sum_id_E(Q_{k,i})=\sqrt{2}$, from which by Lemma \ref{Hs} we get that these measures are finite, too.

Let $\pi:\p^1\to \R$ be the projection $\pi(x,t)=t$. Then 

\begin{equation}\label{F}
\H^2(\pi(F))=\H^2(F)\ \text{for every compact set}\ F\subset\ E.
\end{equation}
The inequality $\H^2(\pi(F))\leq \H^2(F)$ holds, since $\pi$ is 1-Lipschitz. To check the opposite inequality, observe first that $d(R_{k,i})=t_k\sqrt{r_k/2}$, where $\lim_{k\to\infty}t_k=1$. The intervals $\pi(R_{k,i}), i=1,\dots,2N_{k-1},$ cover $\pi(E)$ and they have disjoint interiors, so if $\e>0$, then for large enough $k$ we can cover $F$ with $M_k$ rectangles $R_{k,i}, i\in I_k,$ such that 
$\H^2(\pi(F))+\e\geq M_{k}r_k/2$. Then
$$\H^2(F)\leq \liminf_{k\to\infty}\sum_{i\in I_k}d(R_{k,i})^2 \leq \liminf_{k\to\infty}M_{k}t_k^2r_k/2\leq \liminf_{k\to\infty}t_k^2(\H^2(\pi(F))+\e),$$
from which $\H^2(F)\leq \H^2(\pi(F))$ follows.

Define rectangles $R_{k,i}'\subset R_{k,i}$ so that if $R_{k,i}=I\times [b,b+r_k/2]$, then $R_{k,i}'=I\times [b+r_k/2-r_k^{3/2},r_k/2]$. 
 The sets $\pi(E\cap R_{k,i})$ and $\pi(E\cap R_{k,i}')$ are intervals of lengths $r_k/2$ and $r_k^{3/2}$, whence by \eqref{F} $\H^2(E\cap R_{k,i}')=2\sqrt{r_k}\H^2(\pi(E\cap R_{k,i}))$. Set
$$E'=\bigcap_{l=1}^{\infty}\bigcup_{k=l}^{\infty}\bigcup_{i=1}^{N_k}E\cap R_{k,i}'.$$
We have for every $l$,
$$\H^2(E')\leq \sum_{k=l}^{\infty}\sum_{i=1}^{N_k}\H^2(E\cap  R_{k,i}')=\sum_{k=l}^{\infty}\sum_{i=1}^{N_k}2\sqrt{r_k}\H^2(E\cap  R_{k,i})
=\sum_{k=l}^{\infty}2\sqrt{r_k}\H^2(E).$$
As $\sum_k \sqrt{r_k}<\infty$, we obtain  $\H^2(E')=0$.

Let $a\in E\setminus E'$. Then there is $l$ such that for every $k\geq l$, $a\in R_{k,i_k}'$ for some $i_k$. Let $0<r<1, 0<\delta<1$ and let $k$ be such that $r_{k+1}/\delta < r \leq r_{k}/\delta$. For large enough $k$, $r^2\leq (r_k/\delta)^2<r_k^{3/2}$. Then $B(a,r)$ does not meet any other $R_{k,i}$ except $R_{k,i_k}$. Let $T$ be the vertical axis. As $r_{k+1}<\delta r$, it follows that
$$E\cap B(a,r)\cap \{p:d(p-a,T)\geq\delta r\})=\emptyset.$$
Recalling Lemma \ref{apprtandist} $T$ is the approximate tangent line of $E$ at $\H^2$ almost all points of $E$ and thus $E$ is vertically parabolic  $2$-rectifiable.
\end{proof}

By Remark \ref{graphrem}(2) horizontal Lipschitz graphs project to measure zero on the $t$-axis. But the projection can be uncountable:

\begin{ex}\label{ex2} 
For $0<m\leq n$ there is a compact horizontally parabolic  $m$-rectifiable $m$-Lipschitz graph $G\subset\Pn$ with $\H^m(G)>0$ such that  $P(G)$ is an uncountable Cantor set, where $P(x,t)=t.$
\end{ex}

\begin{proof} 
Let $F'\subset [0,1]$ be a compact Cantor set with $\mathcal L^1(F')>0$ defined as
$$F'=\bigcap_{k=1}^{\infty}\bigcup_{i=1}^{2^k}I_{k.i}',$$
where inside each $I_{k,i}'$ two intervals $I_{k+1,j_1}'$ and $I_{k+1,j_2}'$ are selected by deleting from the middle of $I_{k,i}'$ an interval of length $r_k'$.  
Associate to these the intervals $I_{k,i}$ such that $d(I_{k,i})=d(I_{k,i}')^4$ and the gaps $r_k$ at the $k$ level satisfy $r_k=r_k'^4$. Let
$$F=\bigcap_{k=1}^{\infty}\bigcup_{i=1}^{2^k}I_{k,i}$$ 
and define $f:F'\to F$ such that $f(x)\in I_{k,i_k}$ if $x\in I'_{k,i_k}$. Then $|f(x)-f(y)|\lesssim |x-y|^4$ for $x,y\in F'$. Define $g:F'\to\{0\}\times\R\subset\p^1$ by $g(x)=(0,f(x))$. Then 
$$\lim_{r\to 0}\sup\{\frac{\|g(x)-g(y)\|}{|x-y|}:|x-y|<r, x,y\in F'\}=0.$$
It follows that $P(G_g)=F$ and for every $L>0$ $G_g$ can be expressed as the union of finitely many horizontal $(1,L)$-Lipschitz graphs. Consequently $G_g$ has the required properties when $m=1$.

For $m>1$ we can again take $\{0\}^{n-m}\times [0,1]^{m-1}\times G$.
\end{proof}

\vspace{1cm}
\begin{footnotesize}
{\sc Department of Mathematics and Statistics,
P.O. Box 68,  FI-00014 University of Helsinki, Finland,}\\
\emph{E-mail address:} 
\verb"pertti.mattila@helsinki.fi" 

\end{footnotesize}

\end{document}